\documentclass[hidelinks,onefignum,onetabnum]{siamart250211}

\usepackage{amssymb}
\usepackage{graphicx}
\usepackage{epstopdf} 
\usepackage{amssymb,latexsym,amsmath}
\usepackage{xcolor}
\usepackage{enumerate}
\usepackage{booktabs}
\usepackage{MnSymbol,bbding,pifont}
\usepackage{comment}
\ifpdf
  \DeclareGraphicsExtensions{.eps,.pdf,.png,.jpg}
\else
  \DeclareGraphicsExtensions{.eps}
\fi
\usepackage{tikz, pgfplots, mathtools}

\def\bigo{{\mathcal O}}
\def\e{{\rm e}}

\usepackage{amsopn}

\usepackage{algorithm}
\usepackage{algpseudocode}

\newsiamremark{remark}{Remark}
\newsiamremark{assum}{Assumption}
\newsiamremark{ex}{Example}

\headers{Nearest singular matrix-valued functions}{M. Gnazzo and N. Guglielmi}

\title{On the numerical approximation of the distance to singularity for matrix-valued functions\thanks{Submitted to the editors on December 20, 2023.
\funding{Nicola Guglielmi acknowledges that his research was supported by funds from the Italian MUR (Ministero dell'Universit\`a e della Ricerca) within the PRIN 2017 Project ``Discontinuous dynamical systems: theory, numerics and applications'' and PRIN 2021 Project ``Advanced numerical methods for time dependent parametric partial differential equations with applications''. The work of Miryam Gnazzo was partially supported 
by the PRIN 2022 Project ``Low-rank Structures and Numerical Methods in Matrix and Tensor Computations and their Application''.}}}

\author{Miryam Gnazzo\thanks{Corresponding author. Department of Mathematics, University of Pisa, Pisa, I-56127 (\email{miryam.gnazzo@dm.unipi.it}).}
\and Nicola Guglielmi\thanks{Division of Mathematics, Gran Sasso Science Institute, L'Aquila, I-67100 (\email{nicola.guglielmi@gssi.it}).}}

\begin{document}

\maketitle

\begin{abstract}
Given a matrix-valued function $\mathcal{F}(\lambda)=\sum_{i=1}^d f_i(\lambda) A_i$, with complex matrices $A_i$ and $f_i(\lambda)$ entire functions for $i=1,\ldots,d$, we discuss a method for the numerical approximation of the distance to singularity of $\mathcal{F}(\lambda)$. The closest singular matrix-valued function $\widetilde{\mathcal{F}}(\lambda)$ with respect to the Frobenius norm is approximated using an iterative method. The property of singularity on the matrix-valued function is translated into a numerical constraint for a suitable minimization problem. Unlike the case of matrix polynomials, in the general setting of matrix-valued functions the main issue is that the function $\det ( \widetilde{\mathcal{F}}(\lambda) )$ may have an infinite number of roots. An important feature of the numerical method consists in the possibility of addressing different structures, such as sparsity patterns induced by the matrix coefficients, in which case the search of the closest singular function is restricted to the class of functions preserving the structure of the matrices.
\end{abstract}



\section{Introduction}
Nonlinear matrix-valued functions and the eigenvalue problems associated with them may arise in several scientific contexts. For instance, we find them in the areas of acoustic, fluid mechanics and control theory. In a general framework, given a subset $\Omega \subseteq \mathbb{C}$, a matrix-valued function is a map $F: \Omega \mapsto \mathbb{C}^{n \times n}$. A classical problem consists in the computation of the eigenvalues and the associated eigenvectors, which is usually denoted by nonlinear eigenvalue problem (NEP). There exist several techniques for the resolution of the NEP, such as generalized Newton's method or linearization approaches. A complete survey on this class of functions and a detailed explanation of the methods for the computation of their eigenvalues and eigenvectors may be found in \cite{GutTis}. Nevertheless, the construction of a class of linearization in the sense of \cite{DopicoZaballa} is not possible for each kind of matrix-valued function, due to the different classes of nonlinearities. Therefore, the NEP may not always be reduced into a linear one. For the specific case of polynomial nonlinearities, the problem reduces to consider matrix polynomials of degree $d$, that is:
\begin{equation}
\label{eq:matrix_polynomial}
    P(\lambda)= \sum_{i=0}^d \lambda^i A_i,
\end{equation}
where $A_i \in \mathbb{C}^{n \times n}$, for $i=0,\ldots,d$. 
Another interesting class of nonlinear functions is given by the so called quasipolynomials,
\begin{equation}
\label{eq:m-valued_delays}
    F(\lambda)=\sum_{i=0}^{d_1} \lambda^i A_i + \sum_{j=1}^{d_2} e^{-\tau_j \lambda} B_j,
\end{equation}
with $A_i, B_j \in \mathbb{C}^{n \times n}$, for $i=0,\ldots,d_1$ and $j=1,\ldots,d_2$. 
In particular setting $d_1=1$, with $A_1$ a (possibly singular) matrix, the NEP reduces to solving
\begin{equation*}
    \det \left( F(\lambda) \right)= \det \left( -\lambda A_1 + A_0 + e^{-\tau_1\lambda}B_1 + \ldots + e^{-\tau_{d_2}\lambda} B_{d_2} \right)=0,
\end{equation*}
which corresponds to the general form of the characteristic equation for delay differential algebraic equations with discrete constant delays $\tau_1,\ldots,\tau_{d_2}$ (DDAEs). The eigenstructure of the matrix-valued functions in \eqref{eq:m-valued_delays} is a crucial tool in the solvability of both differential delay algebraic equation systems (DDAEs) and delay differential equations (DDEs). For a detailed overview on the problem, see for instance \cite{MichNiculescu}. In this context, it is important to avoid working with a singular matrix-valued function, such that $\det \left( F(\lambda)\right)$ is identically equal to zero for each $\lambda \in \mathbb{C}$, or also a function which is very close to being singular. For this reason, it would be important to have a method able to approximate the \emph{distance to singularity}. More in detail, given a regular matrix-valued function $F(\lambda)$, such that $\det \left( F(\lambda) \right)$ is not identically zero, we are interested in numerically approximating the nearest matrix-valued function $F(\lambda) + \Delta F(\lambda)$, such that 
\begin{equation*}
    \det \left( F(\lambda) + \Delta F(\lambda) \right) \equiv 0.
\end{equation*}

The computation of an accurate approximation for the distance to singularity of a matrix-valued function has been topic of discussion for many years. The majority of the results are stated for the case of matrix pencils. The first theoretical results for the distance to singularity are proposed in \cite{ByersHeMehr}, where several upper and lower bounds are provided for the case of matrix pencils, but an explicit solution is missing. Very recently, a method that employs a Riemannian optimization framework on the generalized Schur form of pencils has been proposed in \cite{DopicoNoferiniNyman}. An extension of the problem to matrix polynomials of degree greater than $1$ has been proposed in \cite{DasBora}, where the authors prove a characterization of the problem in terms of the rank deficiency of certain convolution matrices associated with the matrix polynomial. Finally, an iterative algorithm based on structured perturbations of block Toeplitz matrices containing the coefficients of the matrix polynomial has been introduced in \cite{GiesHaral}.

Our method for general matrix-valued functions is connected with the work proposed in \cite{GugLubMeh} for matrix pencil and in \cite{GnazzoGugl} for matrix polynomials. A major difficulty is due to the presence of nonlinearities in the matrix-valued function, which represents a delicate point of the problem, since a general matrix-valued function may have an infinite number of eigenvalues. This feature prevents the applicability of the method for matrix polynomials presented in \cite{GnazzoGugl}. The technique for the approximation of the distance to singularity that we present in this article can be adapted to several classes of matrix-valued functions, preserving the advantage of the approach in \cite{GnazzoGugl}, which consists of the possibility to extend the approach to different kinds of structures. 

Similarly to \cite{GugLubMeh}, we propose an ODE-based method, making use of a two level iterative procedure. The most delicate part of the method consists in providing an adequate reformulation of the condition $\det \left(F(\lambda) + \Delta F(\lambda) \right) \equiv 0$, taking into account that it may present an infinite number of zeros in $\mathbb{C}$.

The paper is organized as follows. In Section \ref{sec:Movitating example}, we introduce a motivating example arising in delay differential equations, which stresses the importance of having a method capable to detect the numerical singularity, with particular emphasis on the presence of small delay $\tau$, and, moreover, in Subsection \ref{subsec:overview}, we illustrate the overview of our contribution. Section \ref{sec:problem setting} provides the formulation of the problem. The main notions and definitions are introduced in this section, together with the rephrasing of the problem into an optimization one, and a few results provided by \cite{AustinTref2014} and \cite{TrefWeid} on recall. In Section \ref{sec:Two-level approach}, we propose a new method for the numerical approximation of the distance to singularity for general entire matrix-valued functions. In Section \ref{sec:Extension to structured}, we extend the method to structured matrix-valued functions and specialize the results provided in Section \ref{sec:Two-level approach}. Since the class of matrix polynomials is included in the general class we study, Section \ref{sec:comparison_matrix_pol} provides a new point of view, potentially more efficient, on the computation of the distance to singularity for the polynomial case, with a comparison with the approach in \cite{GnazzoGugl}. A careful analysis of the computational issues connected with the numerical implementation of the new method is provided in Section \ref{sec:Computational issues}, and finally a few numerical examples are provided in Section \ref{sec:Numerical examples matrix valued}.

\section{A motivating example and overview of the contribution}
\label{sec:Movitating example}

In this Section we present an example
from the stability analysis of linear systems of delay differential equations with constant delay.

We consider here initial value problems of delay differential equations:
\begin{equation}\label{eq:dde}
\begin{array}{rcl}
E\, y'(t) &\! = \!& A y(t) + B y(t-\tau) \quad \mbox{for} \quad t \ge 0  \\[2mm]
y(t) & = & g(t) \quad \mbox{for}~~ t \le 0 ,
\end{array}
\end{equation}
where $E, A, B$ are constant $d\times d$ matrices, and the delay $\tau > 0$ is constant.

As an illustrative example, we consider the test problem with  
\begin{align*} \label{eq:linprob}
E & = \begin{pmatrix} 0 &  0 \cr  1 & 1 \end{pmatrix} \qquad
A = \begin{pmatrix} -1 &   \frac12 \cr  0 & -1 \end{pmatrix} \qquad 
B = \begin{pmatrix}  1 &  -\frac12 \cr  0 & \frac12 \end{pmatrix} 
\end{align*}
with initial data (for $t \le 0$)
\begin{align*}
g_1(t) & = \cos(\pi t), \qquad
g_2(t) = 2 - 4 t^2. 
\end{align*}
As for the delay we consider two cases:
\begin{itemize}
\item[(a) ] $\tau=1$; \hskip 1cm (b)  $\tau=10^{-5}$
\end{itemize}
i.e. a constant delay with moderate size and a small delay.

Note that $E$ is singular which implies that the system has differential-algebraic form.
At $t=0$ the first (algebraic) equation yields
\begin{equation} \label{eq:alg}
{\rm eq}(\tau) = -y_1(0) + \frac12 y_2(0) + y_1(-\tau) -\frac12 y_2(-\tau) 
\end{equation}
which is satisfied exactly for $\tau=1$ and to second order 
for small delay $\tau$,
\[
{\rm eq}(\tau) = \left(2-\frac{\pi ^2}{2}\right) \tau ^2+ \bigo \left(\tau ^3\right)
\]
that is with a tiny error, for $\tau=10^{-5}$.

A crucial point, when $E$ is singular, is the well-posedness of the problem, which
relies on the application of the implicit function theorem.
This is equivalent to ask that the matrix pencil $(E,A)$ is regular, which is easy
to check and which is robust with respect to sufficiently perturbations of $A$ and $E$.

We integrate the problem numerically in the two situations (a) and (b) and consider
relative perturbations of size $\bigo(10^{-6})$ on the first entry of $A$ and $B$, 
that is $a_{11}$ and $b_{11}$.
For this collocation methods based on Radau nodes can be successfully applied to stiff delay
differential equations (see \cite{GH01},\cite{GH08}), which are the basis of the code Radar5 
solving stiff and implicit problems.
The results are the following. We consider for example the randomly chosen perturbed entries
\[
\widetilde a_{11} = a_{11} + \delta_1, \qquad \delta_1 =  2\,\cdot 10^{-6}, \qquad
\widetilde b_{11} = b_{11} + \delta_2, \qquad \delta_2 =  2\,\cdot 10^{-6}
\]
We denote by $\widetilde y$ the solution of the perturbed problem.
In the first case ($\tau=1$) we observe that the problem is well conditioned; 
we plot the solution of the original problem (left picture) and the error 
$\mathrm{err}(t) = y(t) - \widetilde y(t)$ (right picture) in Figure \ref{fig:1}.
\begin{figure}[h!]
    \centerline{
    \includegraphics[scale=0.36]{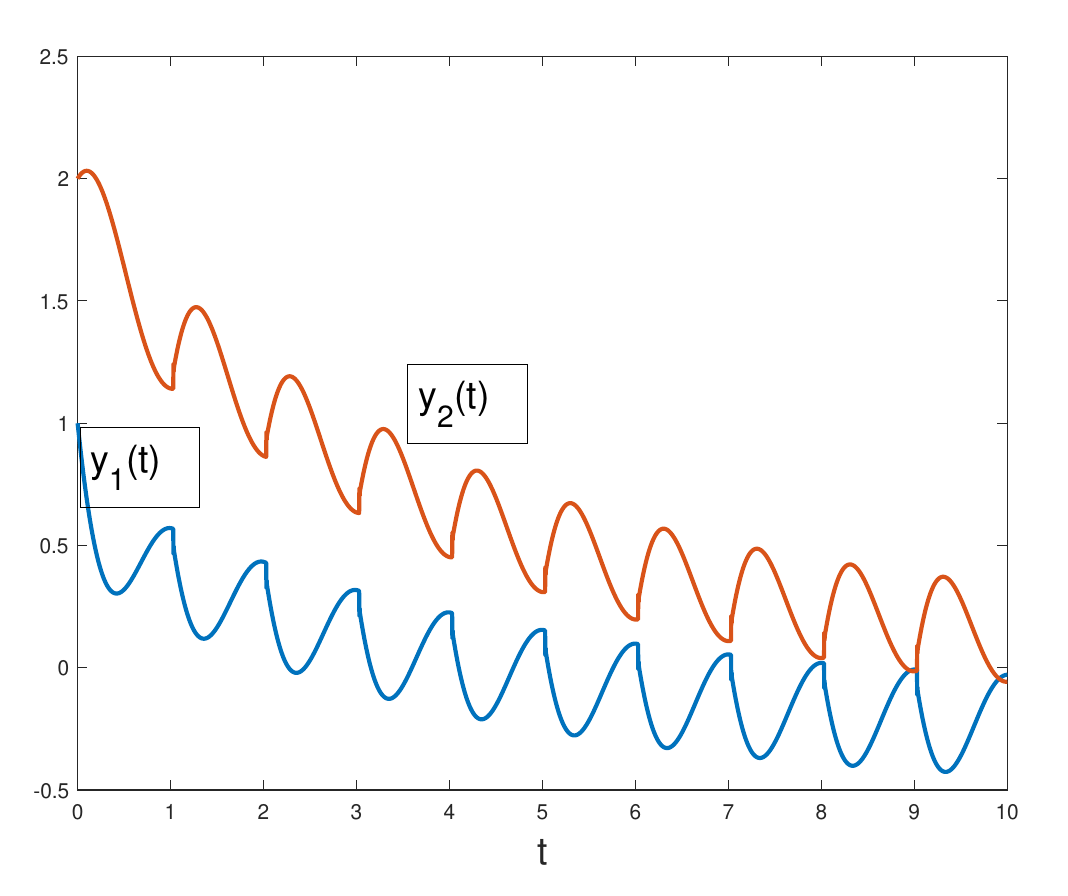} \hskip 1mm 
    \includegraphics[scale=0.36]{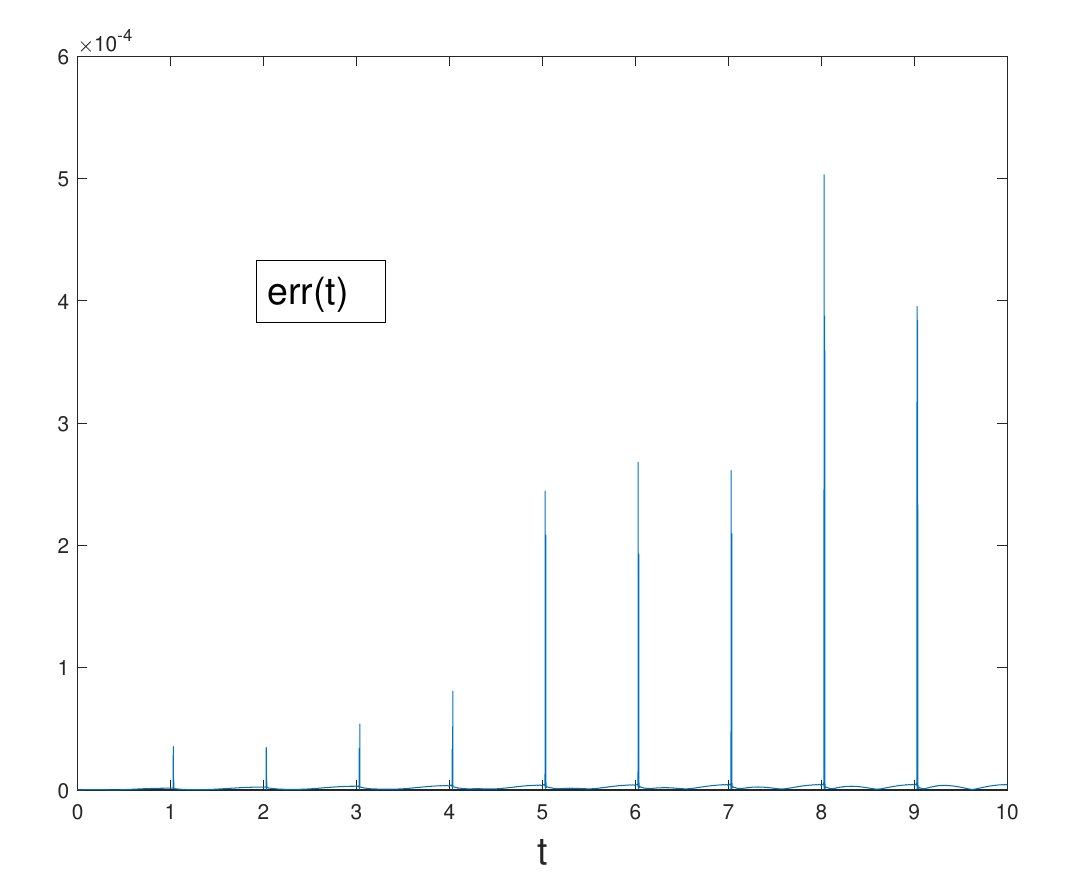}}
    \caption{Plot of the solution of the original problem (left) and error wrt perturbed problem (right) 
    in the case of delay $\tau=1$.}
    \label{fig:1}
\end{figure}
The error at the final point is ${\rm err} = 7.8\,\cdot 10^{-6}$, while the peaks in correspondence
of breaking points $\xi_k = k\,\tau$ have order $\bigo \left( 10^{-4} \right)$.
In the second case ($\tau=10^{-5}$) we observe instead a severe ill conditioning; we plot the solution 
of the original problem $y$ and the solution of the perturbed problem $\widetilde y$ in Figure \ref{fig:2}. 
It is evident that the problem is extremely sensitive to tiny perturbations.
\begin{figure}[h!]
    \centerline{
    \includegraphics[scale=0.36]{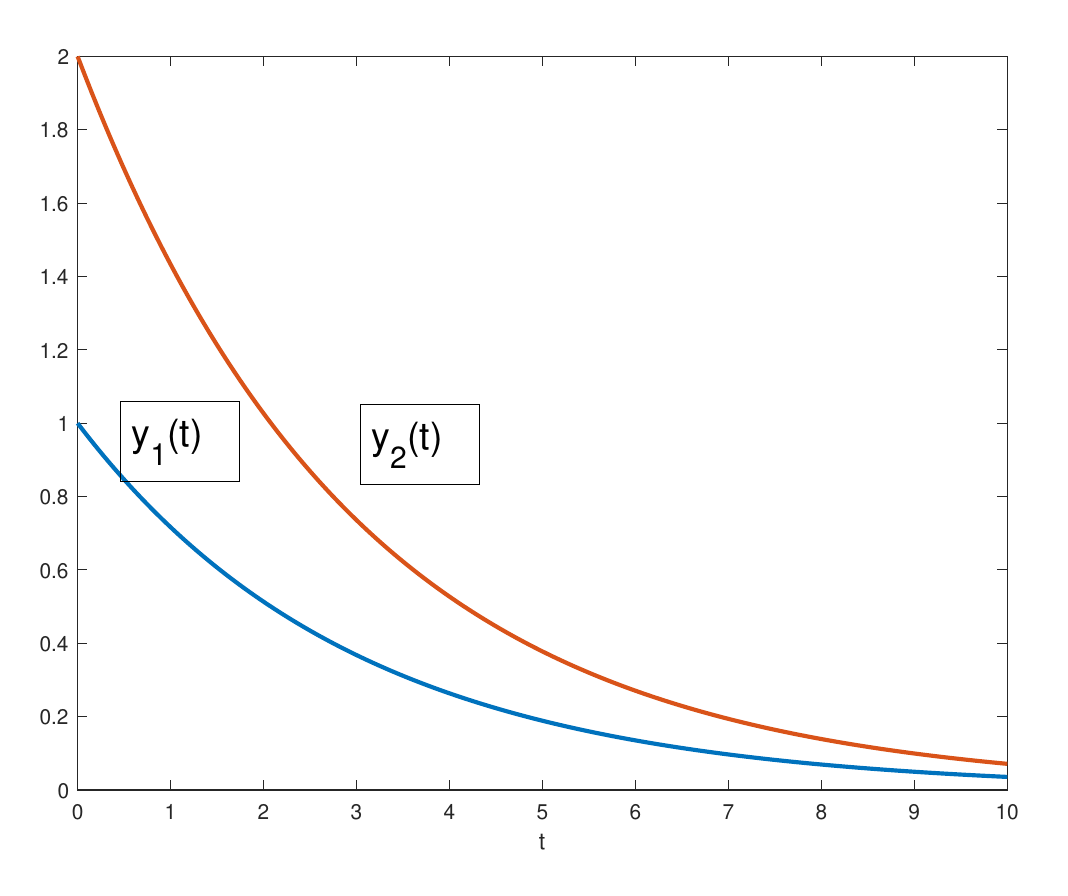} \hskip 1mm 
    \includegraphics[scale=0.36]{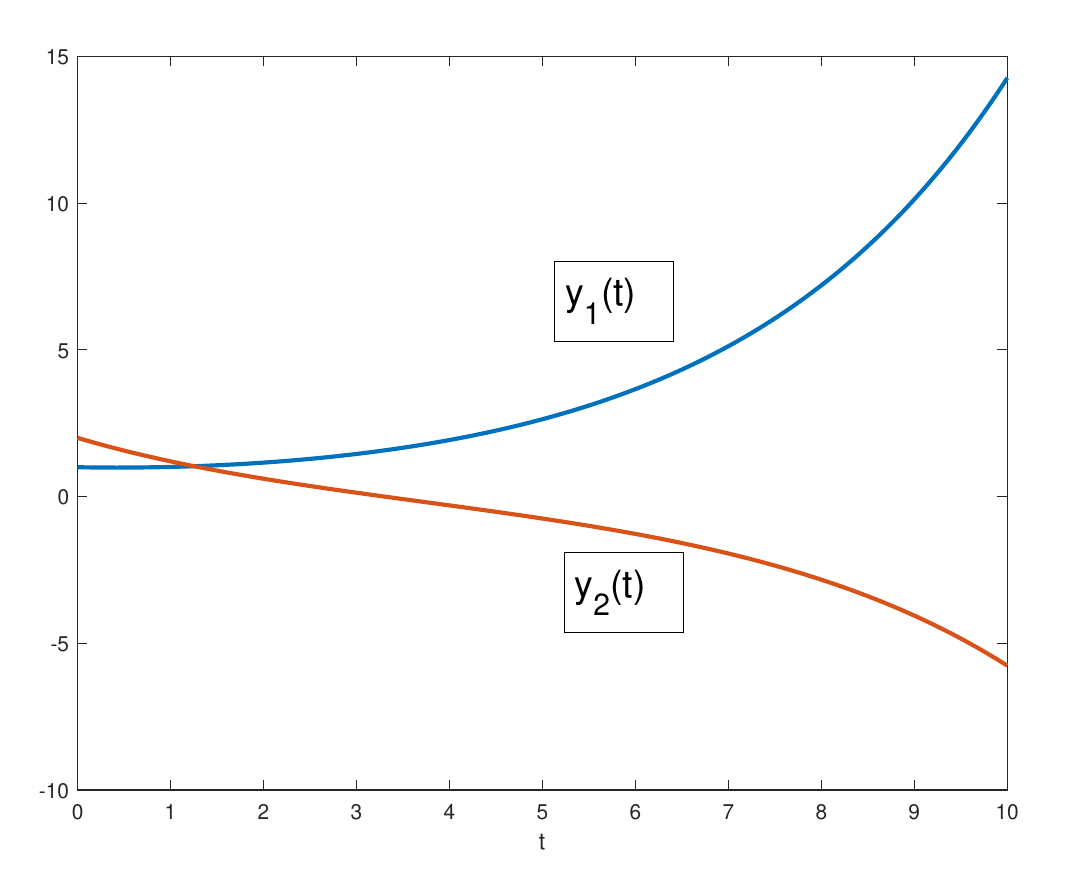}}
    \caption{Plot of the solution of the original problem (left) and perturbed problem (right) 
    in the case of small delay.}
    \label{fig:2}
\end{figure}

In the case of delay $\tau=1$ we observe that the relative error
\[
\varrho = \frac{\| y - \widetilde y\|}{\| \delta \|}, \qquad \delta = \left( \delta_1 \quad \delta_2 \right)^\top
\]
remains nicely bounded. On the contrary, when $\tau=10^{-5}$, the error gets extremely large.

\subsubsection*{Interpretation} Since the pencil is robustly regular, the apparent ill conditioning
of the problem has to be determined by the delay, which is the only difference in the two considered
examples. We let
\[
F(\lambda;\tau) = \det\left( \lambda E - A - B \e^{-\lambda \tau} \right) =
-\frac{3}{2} \lambda \e^{-\tau  \lambda}+\frac{1}{2} \e^{-2 \tau  \lambda}-
 \frac{3 \e^{-\tau  \lambda}}{2}+\frac{3 \lambda}{2}+1.
\]
For given $\tau$, the roots of the function $F(\lambda;\tau)$ are given by $\lambda_0 = 0$ and
the other roots having negative real part, through the Lambert $W$-function. For a complete survey on the Lambert $W$-function, we refer the reader to the work by Corless et al. \cite{Corless}.
It is important to 
notice that the system is stable.
\begin{figure}[h!]
    \centerline{
    \includegraphics[scale=0.65]{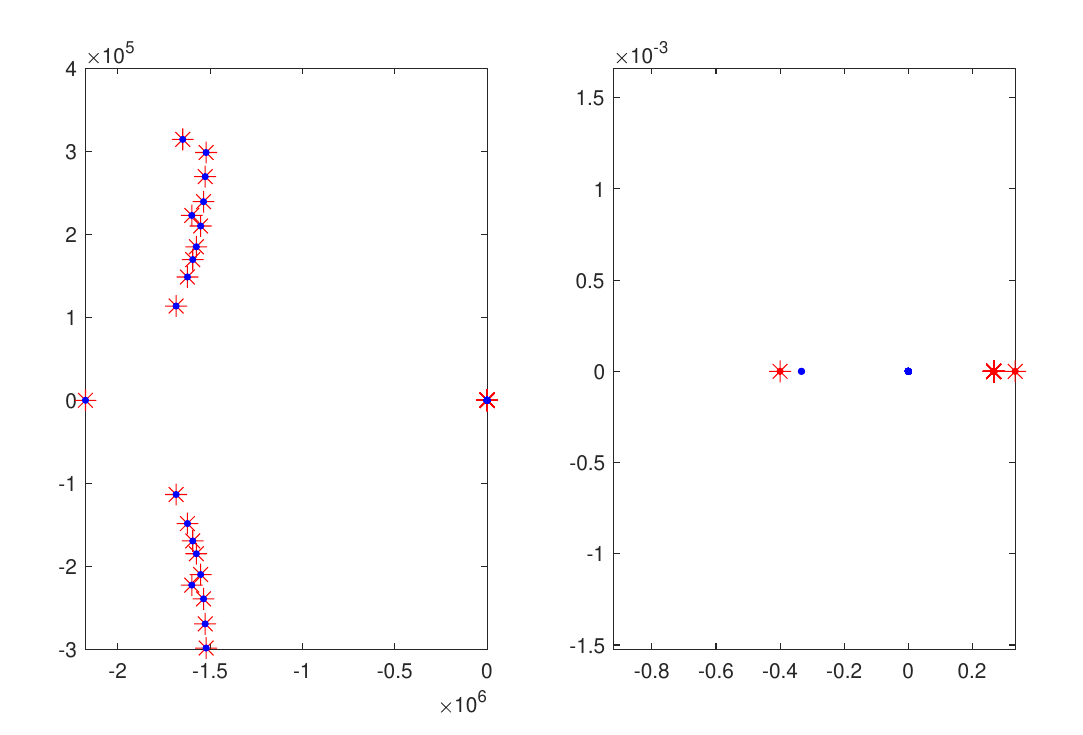}}
    \caption{Plot of the eigenvalues of the original problem (blue points) and perturbed problem (red asterisks) 
    in the case of small delay $\tau$. Right picture: zoom of the rightmost eigenvalues.}
    \label{fig:3}
\end{figure}
Algebraic manipulation provides expression for the roots in closed form. They are given by $\lambda=0$
and 
\begin{eqnarray*} \nonumber
& \lambda_k     = \displaystyle \frac{1}{\tau} W_{k} \left(\frac{\tau}{3} 
          \e^{\frac{2\tau}{3}}\right)-\frac{2}{3} & \text{ if } k \in \mathbb{Z},
\end{eqnarray*}
where $W_k(z)$ is the $k$-th branch of the Lambert $W$-function.

For $\tau=1$, the rightmost real roots are $\lambda=0$ and 
\[
\lambda_0=\frac{1}{3} \left(3 W\left(\frac{e^{2/3}}{3}\right)-2\right) \approx -0.242048,
\]
being $W(z)$ the principal branch of the Lambert $W$-function.

For $\tau=10^{-5}$, the rightmost real roots are $\lambda=0$ and $\lambda_0 \approx -0.333332$.

Note that as soon as we perturb the first entry of matrices $A$ and $B$, it is not possible
anymore to compute the roots of the characteristic equation in closed form through the Lambert
$W$-function.
In this case we make use of the algorithm developed in  \cite{breda2023practical} for the computation of the rightmost roots. 
The results for the case $\tau=10^{-5}$ are shown in Figure \ref{fig:3}.
For the case with $\tau=10^{-5}$, the rightmost root of the perturbed system is $\widetilde \lambda_0 \approx 0.3333\ldots$, and has originated by a perturbation of the eigenvalue $\lambda_0=0$ of the original problem, after a perturbation of magnitude $10^{-6}$. This explains the high ill-conditioning of the problem.
Note that when the delay is $\tau=1$, the perturbation of the rightmost eigenvalue has the same order of the
perturbation. In fact in this case $\widetilde \lambda_0 \approx 4 \cdot 10^{-6}$.

Indeed, when $\tau = 10^{-5}$, one can check directly that
\[
|F\left( \lambda; \tau  \right)| \ll 1 
\]
in a very large region of the complex plane and the eigenvalue problem
appears {\em numerically singular}.
This means that the eigenvalue problem is very close to being singular.
As a consequence, a very tiny perturbation as the one we considered previously in the
example, is able to move the roots significantly and even make the problem unstable, as
it appears in Figure \ref{fig:2} (right picture).
Indeed for $\lambda$ such that $|\lambda \tau| \ll 1$, we have 
\begin{equation} \nonumber
A + B \e^{-\lambda \tau} \approx A + B = \left(
\begin{array}{cc}
 0 & 0 \\
 0 & -\frac{1}{2} \\
\end{array}
\right) 
\end{equation}
and the pair $(E,A+B)$ turns out to be singular. This explains the observed instability.

\subsection{Overview of the contribution}
\label{subsec:overview}

The problem we deal with has not been previously investigated in the literature except
for the case of matrix polynomials. 
The main idea in the article is the following. Given an analytic matrix valued function $F(\lambda)$, 
its singularity is determined by the property that the function 
$f(\lambda) = \det\left(F(\lambda)\right)$  vanishes 
on a closed curve (which is guaranteed by the maximum modulus Theorem). Then, considering an interpolating polynomial for the scalar function $f(\lambda)$ we
are able to get error estimate for $f(\lambda)$ restricted to the curve.

Considering in particular a circle $\mathcal{C}$, in our approach we replace  the condition   
\[
f(\lambda) = 0 \qquad \forall \lambda \in \mathcal{C}
\]
with the weaker interpolation condition 
\[
f(\lambda_k) = 0 \qquad \forall k=1,\ldots,m
\]
at $m$ equidistant points $\{\lambda_k\}$ on the circle $\mathcal{C}$.
For this we are able to exploit the well-known spectral convergence property of the interpolating polynomial for a holomorphic function on a disk, which allows us to make use of a relatively small number $m$ of support points for the interpolation.
In the case when $F$ is a matrix polynomial of degree $n$ ($f$ would be a polynomial of degree $d n$ with $d$ the size of the matrices), the error would be guaranteed to be zero if the number of interpolation points $m > dn$ (as a consequence 
of the fundamental theorem of algebra).
This was the basis of the approach proposed in \cite{GnazzoGugl}.

The problem we are considering here might be tackled differently, approximating the matrix valued function $F$ by a matrix polynomial, with the choice of the points and of the degree to be computed by some suitable criterion, which would be the most delicate part of the method. 
Consider for example the function considered in previous section, $F(\lambda) = \lambda E - A - B \e^{-\lambda}$ with $A,B,E$ given $d \times d$ matrices.
When checking the singularity of $F$,
\[
F(\lambda) = \lambda E - A - B \e^{-\lambda}
\]
we may interpolate the exponential function and replace $F$ by
\[
P(\lambda) = \lambda E - A - B q_n(\lambda) =
( - A + q_{n,0} B) + (E - q_{b,1} B) \lambda - q_{n,2} B \lambda^2 - \ldots - q_{n,n} B \lambda^n
\]
with $q_n(\lambda)$ either the Taylor polynomial or 
an interpolation polynomial of $\e^{-\lambda}$ of degree $n$. 
Using the approach proposed in \cite{GnazzoGugl} we should set the number
of interpolation points at least  $n d +1$; however at such points $P$
would differ from $F$ because the interpolation condition imposed to $q_n$ 
would not guarantee the condition $F(\lambda_k) = P(\lambda_k)$ for $k=1,\ldots nd +1$
so that we cannot replace the problem by an equivalent one for a matrix polynomial. 
Moreover, in the computation of the distance to singularity we should solve a constrained 
optimization problem because all terms of the polynomial $q_m$ would be multiplied by 
the same matrix $B + \Delta B$ (being $\Delta B$ a suitable perturbation of $B$). 
For such an approach the choice of the degree $n$ (and also on the number of support points 
$m$) does not seem easily solvable a priori.

In our perspective, instead, there is no need of replacing $e^{-\lambda}$ by a polynomial and the choice
of $m$ is done by computing estimates of the remainder of the power series of $f$, by suitably approximating 
the associated  Cauchy integral, as we will explain in Section \ref{sec:Computational issues}. 
For this we believe that the proposed approach is effective, as shown by the numerical experiments we provide 
in Section \ref{sec:Numerical examples matrix valued}.

As an important by-product, for the special case when $F$ is a matrix polynomial, we have also improved the 
methodology proposed in \cite{GnazzoGugl} (where $m = nd+1$ was chosen according to the fundamental theorem of 
algebra). In Subsection \ref{subsec:special case_matrixpol}, we have considered a few examples and observed for the new
method a significant improvement in terms of number of interpolation points.

Finally we remark that in practice - for numerical convenience - our method replaces 
$f(\lambda)$ by $\sigma_{\min} \left(F(\lambda)\right)$.

\section{Problem setting}
\label{sec:problem setting}
We use the following notation: given two rectangular $A,B \in \mathbb{C}^{n_1 \times n_2}$, we denote by
\begin{equation*}
    \left\langle A, B \right\rangle = \mbox{Re} \left( \mbox{trace}\left( A^H B \right) \right)
\end{equation*} 
the real Frobenius inner product on $\mathbb{C}^{n_1 \times n_2}$. The associated Frobenius norm is
\begin{equation*}
    \| A \|_F= \left( \sum_{i=1}^{n_1} \sum_{j=1}^{n_2} \left| A_{i,j}\right| ^2 \right)^{\frac{1}{2}}.
\end{equation*}
Consider the matrix-valued function
\begin{equation}
\label{eq:matrix-valued function}
    \mathcal{F}\left( \lambda \right)= f_d\left(\lambda \right) A_d+ \ldots + f_1\left(\lambda\right) A_1,
\end{equation}
where $A_i \in \mathbb{C}^{n \times n}$, for $i=1,\ldots,d$ and the functions $f_i: \mathbb{C} \mapsto \mathbb{C}$ are entire, i.e. analytic on the whole complex plane $\mathbb{C}$, for all $i=1,\ldots,d$. We implicitly assume that $f_d\left(\lambda\right) \not\equiv 0$ and $A_d \neq 0$.

\begin{remark}
    In this work, we restrict the description to matrix-valued functions that are already given in split form \eqref{eq:matrix-valued function}. In most of the applications, indeed, this is the usual representation, see for instance the examples proposed in the collection \cite{NLEVP}. Moreover, the coefficient matrices $F_j$ often represent specific features of a problem, such as the mass matrix or the stiffness matrix. Then, it would be desirable to perform the method we propose using the same framework. It is worth noticing that it may happen to have two (or more) analytic functions $f_i$ that are linear dependent. In this situation, it seems convenient to rewrite the matrix-valued function into a basis where the $f_i$ are linear independent, before applying the method we propose.
\end{remark}

We denote by 
\begin{equation}  \label{eq:norm}
    \| \mathcal{F} \|_F:= \| \left[ A_d, \ldots, A_1 \right] \|_F.
\end{equation}

The function $\mathcal{F}\left( \lambda \right)$ is called regular if
\begin{equation*}
    \det \left( \mathcal{F} \left( \lambda \right) \right) \not \equiv 0,
\end{equation*}
otherwise it is called singular. As we have mentioned, an example of this framework is the class of characteristic functions associated to linear systems of delay differential algebraic equations with constant delays. For example
\begin{equation*}
    A_2 \dot{x}(t) + A_1 x(t-\tau) + A_0 =0, \quad A_i \in \mathbb{C}^{n \times n}, \; \mbox{for} \; i=0,1,2.
\end{equation*}
The problem that we consider is stated as follows.

\begin{definition}
Given a regular matrix-valued function $\mathcal{F}\left( \lambda \right)= \sum_{i=1}^d f_i\left(\lambda \right)A_i$, we define the \emph{distance to singularity} as 
\begin{equation*}
    d_{{\rm{sing}}} \left( \mathcal{F} \right)=\min \left\lbrace \|\Delta \mathcal{F} \|_F : \left( \mathcal{F} + \Delta\mathcal{F} \right) \left( \lambda \right) \; {\rm{is\;singular}}
 \right\rbrace.
\end{equation*}
Here we denote by $\Delta\mathcal{F}$ the perturbation
\begin{equation*}
     \Delta\mathcal{F}\left( \lambda \right)= f_d \left( \lambda \right) \Delta A_d + \ldots + f_1 \left( \lambda \right) \Delta A_1,
\end{equation*}
where the coefficient matrices $\Delta A_i \in \mathbb{C}^{n \times n}$, for $i=1,\ldots,d$.
\end{definition}

\begin{remark}
    Since we are not imposing additional structures, the set of perturbations $\Delta \mathcal{F}$ such that the perturbed matrix-valued function $\mathcal{F} + \Delta \mathcal{F}$ is singular is non-empty. Indeed, we can always choose $\Delta A_i = - A_i$ for $i=1,\ldots,d$. In situations where we look for the structured distance to singularity (as explained in Section \ref{sec:Extension to structured}), we assume that the set where we are minimizing is non-empty.

    This observation is also useful in order to prove that the infimum is a minimum. Indeed, the feasible set consists in the intersection of the set $\mathcal{T}_1:=\left\lbrace \Delta \mathcal{F}(\lambda) : \right.$ $\left.\det \left( \mathcal{F}(\lambda) + \Delta \mathcal{F}(\lambda) \right) = 0, \; \forall \lambda \in \mathbb{C} \right\rbrace$ and the set $\mathcal{T}_2:= \left\lbrace \Delta \mathcal{F}(\lambda) : \| \Delta \mathcal{F} \|_F \leq \|\mathcal{F} \|_F  \right\rbrace$. The set $\mathcal{T}_1$ is closed, while the set $\mathcal{T}_2$ is compact. Then, the intersection $\mathcal{T}_1 \cap \mathcal{T}_2$ is compact. Since the Frobenius norm is a continuous function, the infimum is a minimum by the Weierstrass theorem. Note that in the subsequent Definition \ref{def:structured_distance_matrix_functions}, the proof that the infimum is indeed a minimum can be addressed with the same argument, since we are considering as structures only finite dimensional linear subspaces.
\end{remark}

The problem consists in finding the smallest perturbation $\Delta\mathcal{F}$, in the sense of the Frobenius norm
\eqref{eq:norm}, such that
\begin{equation}
\label{eq:f+deltaf singular function}
    f\left( \lambda \right):= \det \left( \mathcal{F} \left( \lambda \right) +  \Delta\mathcal{F} \left( \lambda \right) \right) \equiv 0.
\end{equation}
Each entry of the matrix-valued function $\mathcal{F}(\lambda) + \Delta \mathcal{F}(\lambda)$ is a combination of the scalar analytic functions $f_1(\lambda), \ldots, f_d(\lambda)$. Thus, the determinant $f(\lambda)$ in \eqref{eq:f+deltaf singular function} is analytic.

\subsection{Reformulation of the problem}
\label{subsec:Reformulation}

Our approach is based on rephrasing the problem of computing the nearest singular matrix-valued function in the form $\mathcal{F}\left( \lambda \right) + \Delta \mathcal{F}\left(\lambda \right)$ into a suitable optimization problem. In order to solve this optimization problem:
\begin{align}
\label{eq:opt_prob_contin}
    \Delta \mathcal{F}^*:= \left[ \Delta A_d^*,\ldots, \Delta A_1^* \right] =& \arg \min_{\Delta \mathcal{F}} \| \Delta \mathcal{F} \|_F, \\
    & \mbox{subj. to} \; f(\lambda)\equiv 0, \notag
\end{align}
we need to replace the constraint $f(\lambda) \equiv 0$ by a discrete setup. In the case of matrix polynomials in the form $\mathcal{F}\left(\lambda\right) = \lambda^{d-1} A_d + \ldots + A_1$, a natural idea comes from the observation that the determinant of a matrix polynomial of degree $d-1$ and size $n$ is a scalar polynomial of degree at most $n \left(d-1\right)$ in the variable $\lambda$. Then, the application of the fundamental theorem of algebra leads to the equivalent condition
\begin{equation*}
    \det \left( \mathcal{F} \left( \mu_j \right) + \Delta \mathcal{F}\left( \mu_j \right) \right)=0,
\end{equation*}
for a prescribed set of distinct complex points $\left\lbrace \mu_j \right\rbrace$, for $j=1, \ldots, n\left(d-1\right) +1$ (for more details, see \cite{GnazzoGugl}). Unfortunately, this approach does not work for the case of general matrix-valued functions. As an illustrative example, the entire function $\det ( \lambda A_2 + \allowbreak e^{-\lambda} A_1 + A_0 )$ does not present a finite number of zeros and the general problem we consider may produce an infinite number of eigenvalues.

In order to replace the condition
\begin{equation*}
\det \left( \mathcal{F}\left( \lambda \right) + \Delta \mathcal{F}\left( \lambda \right) \right) \equiv 0
\end{equation*}
with a more manageable one, we recall a few classical results for holomorphic functions, starting with the following result, which is a consequence of the maximum modulus principle (see for example \cite{Cartan}).

\begin{theorem}
\label{thm:corollary_max_modulus}
Consider $D$ a bounded non empty open subset of $\mathbb{C}$ and $\bar{D}$ the closure of $D$. Suppose that $f: \bar{D} \mapsto \mathbb{C}$ is a continuous function and holomorphic on $D$. Then $\left| f\left( \lambda \right) \right|$ attains a maximum at some points of the boundary of $D$.
\end{theorem}

Since the function $f(\lambda)$ in \eqref{eq:f+deltaf singular function} is entire, in particular it is holomorphic on the set $D:=\left\lbrace \lambda \in \mathbb{C}: \left| \lambda \right| \leq 1 \right\rbrace$. From Theorem \ref{thm:corollary_max_modulus}, we get that the maximum of $f\left( \lambda \right)$ is obtained on the boundary of the unit disk, that is $\left\lbrace \lambda \in \mathbb{C}: \left| \lambda \right| =1 \right\rbrace$. We are interested in finding a perturbation $\Delta\mathcal{F}$ such that
\begin{equation*}
    \max_{\lambda \in \partial D} \left| f(\lambda) \right|=0,
\end{equation*}
from which, consequently, we have that
\begin{equation}
\label{eq:max_unit_disk}
    f(\lambda) \equiv 0 \qquad \lambda \in D.
\end{equation}
Moreover, since $f(\lambda)$ is entire, we have that it can be written using Taylor expansion in $\lambda=0$ as follows:
\begin{equation*}
 f(\lambda)  = \sum_{k=0}^{\infty} \frac{{f}^{\left( k \right)}\left( 0 \right)}{k!} \lambda^k.
\end{equation*}
The property in \eqref{eq:max_unit_disk} yields $f^{\left( k \right)}\left( 0 \right)=0$ for each $k$, which implies $f(\lambda) = 0$ for each $\lambda \in \mathbb{C}$.

These observations lead to a different reformulation of the optimization problem \eqref{eq:opt_prob_contin}, which is the following:
\begin{align}
\label{eq:opt_pbm_unit_disk}
    \left[\Delta A_d^*,\ldots, \Delta A_1^* \right]=&\mbox{arg} \min_{
    \Delta\mathcal{F}} \| \Delta \mathcal{F} \|_F, \\
    &\mbox{subj. to} \; f(\lambda)=0, \;\mbox{for} \; \left| \lambda \right|=1, \notag
\end{align}
where we impose the condition $f(\lambda)=0$ only at $\lambda$ on the boundary of the unit disk.

However, the formulation \eqref{eq:opt_pbm_unit_disk} does not modify the constraint on the determinant in \eqref{eq:f+deltaf singular function} into a discrete one, unlike the analogous for matrix polynomials. Indeed, as explained in \cite{GnazzoGugl}, the fundamental theorem of algebra is able to convert the condition into a discrete one.

The choice of a suitable number of points in the case of general matrix-valued functions represents a delicate feature of our method. Our idea consists in proposing an efficient approximation of the function $f(\lambda)$ through a polynomial interpolant. To this purpose, we use a classical approximation result for analytic functions, which can be found in \cite{AustinTref2014}, and which we report here for completeness.

\begin{theorem}
\label{thm:Thref_approx_pol}
Let $f$ be analytic in $D_{R}=\left\{ z \in \mathbb{C}: \left| z\right| \leq R \right\}$ for some $R>1$. Let $p\left(z \right)$ be the polynomial interpolant of degree $m-1$ at the points $z_k=e^{\frac{2\pi i}{m}k}$, for $k=1,\ldots,m$. Then for any $D_{\rho}$, with $1 < \rho < R$, the polynomial approximation has accuracy
\begin{equation*}
    \left| p(z)-f(z)\right| = \left\lbrace \begin{array}{c}
         O\left( \rho^{-m}\right), \quad \left|z\right| \leq 1,  \\[3mm]
         O\left( \left| z \right|^m \rho^{-m} \right), \quad 1 \leq \left|z\right| < \rho.
    \end{array} \right.
\end{equation*}
\end{theorem}
Theorem \ref{thm:Thref_approx_pol} proves that the polynomial $p$ obtained by interpolation at $m$ points on the unit disk is able to provide an accurate approximation for the determinant of $\mathcal{F}+\Delta \mathcal{F}$. In detail, given the Taylor expansion of $f(\lambda)=\sum_{k=0}^{\infty} a_k \lambda^k$,  with the coefficients expressed by the Cauchy formula
\begin{equation}
\label{eq:coeff_an}
    a_k:= \frac{1}{2\pi i}\int_{\left| \zeta \right|=1} \zeta^{- k -1} f(\zeta) d \zeta,
\end{equation}
the use of Cauchy's estimates proves that
\begin{equation*}
    \left| a_k \right| = O\left(\rho^{-k}\right), \quad \mbox{as} \; k \rightarrow \infty.
\end{equation*}

Theorem \ref{thm:Thref_approx_pol} suggests that the polynomial approximation $\left| p(z) - f(z) \right|$ for $\left| z \right| \in D_R$ is accurate in the complex unit disk. As explained in \cite{AustinTref2014}, for values $\left|z\right| =1$, the leading term of the approximation error is $\left| a_m \right|$. 
Due to this, we decide to monitor the leading term of the approximation error as a possible indicator of the accuracy of the polynomial interpolation, neglecting the remaining terms in the approximation error. As observed by an anonymous referee, we cannot prove rigorously that $\left| p(z)- f(z) \right|$
is bounded by a tolerance strictly related to the magnitude of $a_m$. For this reason, in our numerical implementation of the method, we consider it jointly with a more robust criterion,
based on the computation of $\left| p(z)- f(z) \right|$ on  suitable grid. 

If $|a_m|$ is large we cannot expect an accurate approximation; on the other side if $|a_m|$ is small enough we further check $\left| p(z)- f(z) \right|$ on a numerical grid, and if this is also small, we choose the value $m$.
This combination of choices guarantees a sufficiently precise polynomial approximation of the function $f(z)$. 
We describe this computational procedure in the subsequent Section \ref{sec:Computational issues}. Moreover, we numerically illustrate this procedure on a set of numerical tests in Section \ref{sec:Numerical examples matrix valued}.

Since the theoretical bound for \eqref{eq:coeff_an} involves the $m$-th derivative of the function $f(\lambda)$, we proceed by using a numerical approximation of the integral at the right-hand side.

We propose a method based on a proper approximation of the $m$-th coefficient of the Taylor expansion in \eqref{eq:coeff_an} for the function $f(\lambda)$. Our idea consists in finding a value $m$ such that the coefficient
\begin{equation*}
    \left| a_m \right| \leq \mbox{tol},
\end{equation*}
where $\mbox{tol}$ should be taken smallest than the desired accuracy of the method. The integrand function in \eqref{eq:coeff_an} presents a pole at $\lambda =0$, then it is still analytic in a annulus centered in the origin. In order to study the convergence of $a_m$, we exploit the properties of the trapezoidal rule for analytic functions, presented in \cite[Theorem 12.1]{TrefWeid}. In particular, we refer to the following theorem:
\begin{theorem}
\label{thm:Trapezoidal_Trefethen}
Suppose $f$ is analytic and satisfies $\left| f(\lambda) \right| \leq M$ in the annulus $r^{-1} < \left| \lambda \right| < r$, for some $r>1$. Consider the Laurent series representation of $f$ 
\[
f(z)=\sum_{j=-\infty}^{+\infty} a_j z^j, \; \mbox{where} \; a_j=\frac{1}{2\pi i} \int_{\left| \zeta\right| =1} \zeta^{-j-1} f(\zeta) d\zeta,
\]
and, for $N \geq 1$, the approximation $a_j^{[N]}$ of the coefficient $a_j$, given by the trapezoidal rule, that is
\[
a_j^{[N]} =\frac{1}{N} \sum_{\ell=1}^N z_{\ell}^{-j}f(z_{\ell}), \; \mbox{where} \; z_{\ell}=e^{\frac{2\pi i}{N}\ell}, \; \mbox{for} \; \ell=1,\ldots,N.
\]
Then we have:
\begin{equation*}
    \left| a_j^{\left[N\right]} -a_j\right| \leq \frac{M \left( r^j + r^
    {-j}\right)}{r^N-1}.
\end{equation*}
\end{theorem}
Theorem \ref{thm:Trapezoidal_Trefethen} holds for the function $f(\lambda)$, which is entire on the complex plane. Then, we have that the approximation $a_j^{\left[N\right]}$ converges to the respective coefficient $a_j$ of the Taylor series exponentially in $N$. This is useful in the numerical implementation of the method, since it allows us to efficiently approximate the Taylor coefficients $a_j$, by means of the trapezoidal rule.

Then, from the results given by Theorems \ref{thm:Thref_approx_pol} and \ref{thm:Trapezoidal_Trefethen}, we choose as set of evaluations points
\begin{equation}
\label{eq: points_from Thefethen}
    \mu_j= e^{\frac{2 \pi i}{m}j}, \qquad j=1,\ldots,m.
\end{equation}

Using the previous remarks, we construct the discrete version of the optimization problem, in the following way:
\begin{align}
\label{eq:opt_pbm_finite_n}
    \left[\Delta A_d^*,\ldots, \Delta A_1^* \right]=&\mbox{arg} \min_{
    \Delta\mathcal{F}} \| \Delta \mathcal{F} \|_F, \\
    &\mbox{subj. to } f( \mu_j )=0, \quad \mbox{for} \; j=1,\ldots,m, \notag
\end{align}
where we choose the points $\mu_j$ as in \eqref{eq: points_from Thefethen}.

\section{A two-level approach}
\label{sec:Two-level approach}

Starting from this Section, we consider matrix-valued function $\mathcal{F}(\lambda)$ such that
\begin{equation*}
    \max_{p \in \Xi} \left| \det ( \mathcal{F}(p)) \right| =1,
\end{equation*}
where $\Xi$ is a discrete set of points and $\Xi \subseteq \partial D$, with $D$ the unit disk in the complex plane. Ideally, we would like to normalize the starting matrix-valued function such that $\max_{\lambda \in \partial D} \left| \det ( \mathcal{F}(\lambda) )\right| =1$, but, instead, we propose to consider a discrete set of points $\Xi$, which we describe in Section \ref{sec:Computational issues}. It is important to remark that this normalization is not necessary from a theoretical point of view, but it may be useful from the point of view of the numerical implementation, since it would allow us to deal with determinant of very large or small magnitude.

The distance to singularity is determined by the solution of the optimization problem \eqref{eq:opt_pbm_finite_n}. As proposed in \cite{GnazzoGugl}, we rephrase the optimization problem into an equivalent one, and solve it by an iterative method. For convenience, we denote the perturbations as
\begin{equation*}
    \left[ \Delta A_d, \ldots, \Delta A_1 \right]= \varepsilon \left[ \Delta_d, \ldots, \Delta_1 \right],
\end{equation*}
where $\varepsilon >0$ and $\|\left[ \Delta_d, \ldots, \Delta_1 \right]\|_F=1$, from which we have that $\| \Delta \mathcal{F} \|_F=\varepsilon$. We aim to develop a numerical method to approximate the smallest $\varepsilon$ such that $\det ( \mathcal{F}(\mu_j) + \allowbreak \Delta \mathcal{F}(\mu_j) )=0$, for $j=1,\ldots,m$. 

\subsection{Numerically singular matrices}

The next step consists in translating the constraint on the singularity of the $m$ evaluations of the perturbed matrix-valued function, i.e.
\begin{equation}
\label{eq:det_at_points_section_numerical_sing_matrices}
    \det \left( \mathcal{F}(\mu_j) + \Delta \mathcal{F}(\mu_j) \right)=0, 
\end{equation}
for each $\mu_j$, with $j=1,\ldots,m$, into a more practicable problem.

Indeed, due to the possible numerical issues connected with the computation of a determinant, we prefer to employ the singular value decomposition of the matrices $\mathcal{F}(\mu_j)+ \Delta\mathcal{F}(\mu_j)$ in the numerical implementation of the method, in order to study the singularity of each matrix for $j=1,\ldots,m$. Therefore, for each $j=1,\ldots,m$ we will take into account the smallest singular value
\begin{equation*}
     \sigma_{\min}\left(\mu_j\right):= \sigma_{\min} \left( \mathcal{F} \left(\mu_j \right) + \Delta \mathcal{F} \left( \mu_j \right) \right),
 \end{equation*}
and set $\sigma_{\min}\left( \mu_j \right)=0$ instead of the determinant \eqref{eq:det_at_points_section_numerical_sing_matrices}.
\begin{remark}
We emphasize that the numerical implementation of our method provides a perturbation $\Delta \mathcal{F}\left(\lambda \right)$ for which we can only guarantee that the condition \eqref{eq:f+deltaf singular function} is satisfied with a tiny error. Therefore, in this setting it may be more appropriate to use the notion of \emph{numerical rank} of a matrix (see \cite{GoluVanl13}, Section $5.4.1$):

\begin{definition}
Consider a matrix $A \in \mathbb{C}^{n \times n}$ and a certain threshold $\delta >0$, larger or equal than machine precision. Let $\sigma_1 \geq \sigma_2 \geq \ldots \geq \sigma_n$ be the singular values computed in finite arithmetic. We say that $r$ is the \emph{numerical rank} of the matrix $A$ if
\begin{equation*}
    \sigma_1 \geq \sigma_2 \geq \ldots \geq \sigma_{r} > \delta \geq \sigma_{r+1} \geq \ldots \geq \sigma_n.
\end{equation*}
Consequently, we define the matrix $A$ \emph{numerically singular} if the numerical rank $r < n$.
\end{definition}
In the subsequent Section \ref{sec:Computational issues}, we shall make use of a sufficiently small threshold $\delta$, which guarantees that the matrices $\mathcal{F}(\mu_j) + \Delta \mathcal{F}(\mu_j)$ are numerically singular for each $\mu_j$, $j=1,\ldots,m$.
\end{remark}
In order to solve the optimization problem \eqref{eq:opt_pbm_finite_n}, we propose a two-level iterative methodology:
\begin{enumerate}[(i)]
    \item an \emph{inner iteration}, which is a minimization procedure at fixed $\varepsilon$ for the functional
    \begin{equation}
    \label{eq:functional_G}
    G_{\varepsilon}\left( \Delta_d,\ldots, \Delta_1 \right):=\frac{1}{2} \sum_{j=1}^m \sigma^2_{j}(\Delta_d,\ldots, \Delta_1),
    \end{equation}
    where we denote with $\sigma_j=\sigma_{\min}(\mu_j)$, for $j=1,\ldots,m$ (we omit for brevity the dependence of the singular values from $\varepsilon$). From this step, we get the local minimizers $\Delta_d (\varepsilon), \ldots, \Delta_1 (\varepsilon)$;
    \item an \emph{outer iteration}: a strategy to searching for the smallest positive zero $\varepsilon^*$ of the functional
    \begin{equation*}
        G_{\varepsilon} \left( \Delta_d (\varepsilon), \ldots, \Delta_1 (\varepsilon) \right) =0.
    \end{equation*}
\end{enumerate}
Note that in the numerical implementation the value of the number of support points $m$ may change as the value of $\varepsilon$ changes. This requires an additional search of the most suitable choice of the number of needed points $m$. For simplicity, in this Section we provide a description of the method at fixed $m$. The details about the choice of the number $m$ of points are addressed in Section \ref{sec:Computational issues}.

\subsection{Inner iteration}
\label{subsec:Inner for matrix-valued functions}

The inner iteration consists in a minimization procedure that computes the local minimizers $\Delta_d (\varepsilon),\ldots, \Delta_1 (\varepsilon)$ of the functional \eqref{eq:functional_G}.  Then, throughout this subsection, we study the problem and minimize the functional \eqref{eq:functional_G} at fixed $\varepsilon$. We propose to solve this optimization problem by using a constrained steepest descent method for the functional $G_{\varepsilon}$, at fixed $\varepsilon$ and $m$. In order to compute the gradient of $G_{\varepsilon}$, we parametrize with respect to the time the perturbations, that is $\Delta_i(t)$, for $i=1,\ldots,d$ and impose the constraint $\| \boldsymbol{\Delta}(t) \|_F=1$, where we denote $\boldsymbol{\Delta}(t):=\left[ \Delta_d(t),\ldots, \Delta_1(t) \right]$. This allows us to employ a continuous in time minimization method, to solve the inner iteration, following the idea presented in \cite{GuglLubich}. Indeed, we aim to compute a trajectory reaching, at a stationary point, a solution of the optimization problem. To do this, we compute the gradient and study the constrained gradient system. The requirement on the norm $\|\boldsymbol{\Delta}(t) \|_F=1$ can be translated as
\begin{equation*}
0=\frac{d}{dt}\| \boldsymbol{\Delta}(t) \|_F^2 = 2 \left\langle \boldsymbol{\Delta}(t), \dot{\boldsymbol{\Delta}}(t) \right\rangle,
\end{equation*}
and then, the condition becomes
\begin{equation*}
    \left\langle \boldsymbol{\Delta}(t), \dot{\boldsymbol{\Delta}}(t) \right\rangle=0.
\end{equation*}
The steepest descent method requires the derivative of the gradient $G_{\varepsilon}$ and therefore, the derivatives of the singular values $\sigma_{j}$, with respect to $t$. From the standard theory of perturbation (see for example \cite{HornJoh}), we have that each simple and nonzero singular value $\sigma_{j}(t)$ is differentiable and, for $j=1,\ldots,m$, the derivative of $\sigma^2_j(t)$ is
\begin{align*}
   \frac{1}{2} \frac{d}{dt}\sigma^2_j= \dot{\sigma}_j \sigma_j &=
   \varepsilon \sigma_j \mbox{Re} \left(u_j^H \left(\sum_{i=1}^d f_i(\mu_j) \dot{\Delta}_i \right) v_j\right) \\
   &= \varepsilon \sum_{i=1}^d \left\langle \sigma_j \overline{f_i(\mu_j)}u_jv_j^H, \dot{\Delta}_i \right\rangle,
\end{align*}
where we dropped the dependence from $t$ and we denote by $u_j,v_j$ the left and right singular vectors of the matrix $\sum_{i=1}^d f_i(\mu_j) \left(A_i + \varepsilon \Delta_i \right)$, associated with the smallest singular value. From these computations, we write the derivative of \eqref{eq:functional_G} with respect to $t$, making use of the Frobenius inner product:
\begin{align}
\label{eq:derivative_G_eps}
    \frac{d}{dt} G_{\varepsilon}\left( \boldsymbol{\Delta}(t) \right)&=
    \varepsilon  \left( \left\langle M_d, \dot{\Delta}_d \right\rangle + \ldots + \left\langle M_1, \dot{\Delta}_1 \right\rangle \right) \\ 
    &= \varepsilon \left\langle \left[ M_d, \ldots, M_1 \right], \left[ \dot{\Delta}_d, \ldots, \dot{\Delta}_1 \right] \right\rangle, \notag
\end{align}
where we define
\begin{equation*}
    M_i:=\sum_{j=1}^m \sigma_j \overline{f_i(\mu_j)} u_j v_j^H,
\end{equation*}
for $i=1,\ldots,d$. From \eqref{eq:derivative_G_eps}, we denote $\mathbf{M}:=\left[ M_d, \ldots, M_1 \right]$ and we have that
\begin{equation}
\label{eq:derivative_functional_matrix-valued}
    \frac{1}{\varepsilon} \frac{d}{dt}G_{\varepsilon}\left( \boldsymbol{\Delta}(t) \right)= \left\langle \mathbf{M}, \dot{\boldsymbol{\Delta}} \right\rangle,
\end{equation}
which identifies $\mathbf{M}$ as the free gradient of $G_{\varepsilon}$. Note that the direction of steepest descent for the functional $G_{\varepsilon}$ must be computed, taking into account both the relation \eqref{eq:derivative_functional_matrix-valued} and the constraint on the admissible set of perturbations, that is $\boldsymbol{\Delta}(t)$ such that $\| \boldsymbol{\Delta}(t)\|_F=1$. Then, this direction of steepest descent for the constrained functional can be obtained using:
\begin{lemma}[Constrained steepest descent method] Consider $\mathbf{M},\boldsymbol{\Delta} \in \mathbb{C}^{n \times dn}$ not proportional to each other and $\mathbf{Z} \in \mathbb{C}^{n \times dn}$. A solution of the constrained minimization problem
\begin{align*}
    \mathbf{Z}^*=&{\rm{arg}} \min_{\mathbf{Z} \in \mathbb{C}^{n \times dn}} \left\langle \mathbf{M}, \mathbf{Z} \right\rangle \\
    &{\rm{subj.\; to}} \; \left\langle \mathbf{Z}, \boldsymbol{\Delta} \right\rangle=0, \\
    &{\rm{and}} \; \| \mathbf{Z} \|_F=1 \; {\rm{(normalization \; constraint)}}
    \end{align*}
is given by
\begin{equation}
\label{eq:solution_minimiz}
    \kappa \mathbf{Z}^*= - \mathbf{M} + \eta \boldsymbol{\Delta},
\end{equation}
where $\eta= \left\langle \mathbf{M}, \boldsymbol{\Delta}\right\rangle$ and $\kappa$ is the norm of the right-hand side.
\label{lem:minim}
\end{lemma}

\begin{proof}
As a first step, we note that the real Frobenius inner product on $\mathbb{C}^{n \times dn}$ is the real Frobenius inner product on $\mathbb{R}^{n \times 2dn}$. Moreover, the real Frobenius inner product on $\mathbb{R}^{n \times 2dn}$ is the standard scalar inner product on $\mathbb{R}^{2dn^2}$, if we consider the vectorizations of the matrices involved. Then, we consider the vectors $z,m,\delta$ that are the vectorizations of the matrices $\mathbf{Z},\mathbf{M},\boldsymbol{\Delta}$, respectively. The constraint minimization problem reduces to:
\begin{align*}
    z^* = \arg \min_{\|z\|_2=1,\; z^T\delta= 0} m^Tz.
\end{align*}
Without the additional constraint that the solution should be in $\left\lbrace z \in \mathbb{R}^{2dn^2}: z^T \delta = 0 \right\rbrace$, we have that 
\begin{align*}
    \arg \min_{\|z\|_2=1} m^Tz = - \frac{m}{\| m \|_2}.
\end{align*}
Incorporating the constraint, we have that the $z^*$ is equal to projecting $-m$ onto the subspace $\left\lbrace z \in \mathbb{R}^{2dn^2}: z^T \delta = 0 \right\rbrace$ and normalizing the result. Indeed, at the end we have that:
\begin{align*}
    z^* = \frac{-m + (\delta^T m) \delta}{\|-m + (\delta^T m)\delta\|_2},
\end{align*}
and we may define $\eta := (\delta^T m)$. Constructing the associated matrices, we have that the statement holds, and the norm conservation follows from $- \langle   \boldsymbol{\Delta},  \mathbf{M}  \rangle + \eta \langle \boldsymbol{\Delta}, \boldsymbol{\Delta} \rangle = - \langle   \boldsymbol{\Delta}, \mathbf{M} \rangle + \eta = 0$. 
\end{proof}

Then, the direction of steepest descent for the functional $G_{\varepsilon}$ is given by the solution of the minimization problem stated in Lemma \ref{lem:minim}. As suggested by Lemma \ref{lem:minim}, we may consider 
the constrained gradient system for the functional $G_{\varepsilon}\left( \boldsymbol{\Delta}\right)$:
\begin{equation}
\label{eq:gradient_system}
    \dot{\boldsymbol{\Delta}}= - \mathbf{M} + \eta \boldsymbol{\Delta},
\end{equation}
with initial datum of unit Frobenius norm. In the following, we prove results that propose a characterization of the stationary points of \eqref{eq:gradient_system} as local extremizers of the functional $G_{\varepsilon}$. Firstly, Theorem \ref{thm:G_decrease} proves that the functional $G_{\varepsilon}$ decreases monotonically along the solution trajectories of \eqref{eq:gradient_system}.   

\begin{theorem}
\label{thm:G_decrease}
Consider $\boldsymbol{\Delta}(t) \in \mathbb{C}^{n \times dn}$, with $\|\boldsymbol{\Delta}(t) \|_F=1$ and solution of the gradient system \eqref{eq:gradient_system}. Consider $\sigma_j(t)$ simple nonzero singular value of the matrix $\sum_{i=1}^d f_i(\mu_j)\left( A_i + \varepsilon \Delta_i(t) \right)$, for $j=1,\ldots,m$. Then:
\begin{equation*}
    \frac{d}{dt}G_{\varepsilon}\left( \boldsymbol{\Delta}(t)\right) \leq 0.
\end{equation*}
\end{theorem}
\begin{proof}
Since $\| \boldsymbol{\Delta}(t) \|_F =1$ , we have that $\left\langle \boldsymbol{\Delta}, \dot{\boldsymbol{\Delta}}\right\rangle = 0$. Then, using the expression of the solution $\boldsymbol{\Delta}$ in \eqref{eq:gradient_system}, we obtain:
\begin{align*}
    \frac{d}{dt} G_{\varepsilon}(\boldsymbol{\Delta}) = \varepsilon \left( \left\langle \mathbf{M}, \dot{\boldsymbol{\Delta}} \right\rangle - \eta \underbrace{\left\langle \boldsymbol{\Delta}, \dot{\boldsymbol{\Delta}} \right\rangle}_{=0} \right) = \varepsilon \left\langle \mathbf{M} - \eta \boldsymbol{\Delta}, \dot{\boldsymbol{\Delta}} \right\rangle = - \varepsilon \| \dot{\boldsymbol{\Delta}} \|_F^2.
\end{align*}
\end{proof}
Moreover, we provide a characterization of the stationary points of the constrained gradient system \eqref{eq:gradient_system} in the following Theorem.
\begin{theorem}
\label{thm:charact_minim}
Consider a solution $\boldsymbol{\Delta}(t)$ with unit Frobenius norm of the gradient system \eqref{eq:gradient_system}. Assume that, for each $t$, we have that $G_{\varepsilon}\left( \boldsymbol{\Delta} \right)>0$ and that the smallest singular values $\sigma_j(t)$ of $\sum_{i=1}^d f_i(\mu_j)\left( A_i + \varepsilon \Delta_i(t) \right)$, for each $j=1,\ldots,m$, are simple and different from $0$. Consider, for $j=1,\ldots,m$, $u_j,v_j$ the left and right singular vectors associated with $\sigma_j$, respectively. Then the following statements are equivalent:
\begin{enumerate}
    \item[(i)] $\frac{d}{dt} G_{\varepsilon}\left( \boldsymbol{\Delta} \right)=0$;
    \item[(ii)] $\dot{\boldsymbol{\Delta}}=0$;
    \item[(iii)] $\boldsymbol{\Delta}$ is a real multiple of the matrix $\mathbf{M}$.
\end{enumerate}
\end{theorem}
\begin{proof} 
It is clear that \emph{(iii)}$\implies$\emph{(ii)}$\implies$\emph{(i)} holds. Then, we prove \emph{(i)}$\implies$\emph{(iii)}. From the formula \eqref{eq:derivative_functional_matrix-valued}, we get that:
\[
\frac{1}{\varepsilon}  \frac{d}{d t} G_{\varepsilon}(\boldsymbol{\Delta}) = - \| \mathbf{M} \|_F^2 + \left\langle \mathbf{M}, \boldsymbol{\Delta} \right\rangle^2 \leq 0,
\]
where we employed the Cauchy-Schwarz inequality and the property $\| \boldsymbol{\Delta}\|_F=1$ in the last step. Moreover, the inequality is strict, unless we have $\boldsymbol{\Delta} =\pm \frac{\mathbf{M}}{\| \mathbf{M}\|_F}$.
\end{proof}

The results described in Subsection \ref{subsec:Inner for matrix-valued functions} suggest to find the local extremizers of the functional $G_{\varepsilon}$ by computing the stationary points of the gradient system \eqref{eq:gradient_system}.

\subsection{Outer iteration}
\label{subsec:Outer for matrix-valued}

The distance to singularity is given by the smallest value $\varepsilon^*$ for which the perturbed matrix-valued function $\mathcal{F} + \Delta \mathcal{F}$ is singular. Therefore, our goal is to solve the minimization problem
\begin{equation*}
    \varepsilon^*= \min \left\lbrace \varepsilon > 0: G_{\varepsilon}\left( \boldsymbol{\Delta}( \varepsilon) \right) =0 \right\rbrace,
\end{equation*}
where $\boldsymbol{\Delta}( \varepsilon)$ is a path of stationary points of \eqref{eq:gradient_system}, which means that for each $\varepsilon$ the matrix $\boldsymbol{\Delta}(\varepsilon)=[\Delta_d(\varepsilon),\ldots, \Delta_1(\varepsilon)]$ is the (local) minimizer obtained in the inner iteration, for the value $\varepsilon$. For simplicity, we denote by $g\left( \varepsilon \right):= G_{\varepsilon}\left( \boldsymbol{\Delta}( \varepsilon) \right)$.

\begin{remark}
The functional $G_{\varepsilon}\left( \boldsymbol{\Delta}( \varepsilon) \right)$ also depends on the number $m=m(\varepsilon)$ of points, needed to reach the chosen accuracy. More details about the choice of $m$ are provided in Section \ref{sec:Computational issues}. Here, it is very reasonable to assume that as the value of $\varepsilon$ approaches the target value $\varepsilon^*$, the number of considered points $m$ remains constant. Therefore, we introduce the following assumption:
\end{remark} 

\begin{assum}
    \label{assumption_points}
For values of $\varepsilon$ sufficiently close to $\varepsilon^*$, the number of points $m(\varepsilon)$ needed in the inner iteration is constant.
\end{assum}

In general, in our numerical experiments the number of points does not present drastic variations and appears to be piecewise constant. In Figure \ref{fig:Plot_points_m(eps)}, we provide a plot of the function $\varepsilon \mapsto m \left( \varepsilon \right)$, for the matrix-valued function in the subsequent Example \ref{ex:different_tau} (choosing $\tau=1$). The function is piecewise constant and assumes values in the interval $\left[ 16,19 \right]$.

\begin{figure}
    \centering
\begin{tikzpicture}
    \begin{axis}[xmin=0, xmax=0.7500,
    xtick={0,0.2, 0.4, 0.7034},
    xticklabels={0,0.2, 0.4, 0.7034},
    xlabel=$\varepsilon$, ylabel=$m(\varepsilon)$,
    ]
        \addplot[thick, color=blue, domain=0:0.1450, samples=400]{19};
        \addplot[thick, color=blue, domain=0.1450:0.2550, samples=400]{18};
        \addplot[thick, samples=50, dashed, blue] coordinates {(0.1450,18)(0.1450,19)};
        \addplot[thick, color=blue, domain=0.2550:0.2700, samples=400]{17};
        \addplot[thick, samples=50, dashed, blue] coordinates {(0.2550,17)(0.2550,18)};
        \addplot[thick, color=blue, domain=0.2700:0.2800, samples=400]{18};
        \addplot[thick, samples=50, dashed, blue] coordinates {(0.2700,17)(0.2700,18)};
        \addplot[thick, samples=50, dashed, blue] coordinates {(0.2800,17)(0.2800,18)};
        \addplot[thick, color=blue, domain=0.2800:0.3000, samples=400]{17};
        \addplot[thick, samples=50, dashed, blue] coordinates {(0.3000,17)(0.3000,18)};
        \addplot[thick, color=blue, domain=0.3000:0.3150, samples=400]{18};
         \addplot[thick, samples=50, dashed, blue] coordinates {(0.3150,17)(0.3150,18)};
        \addplot[thick, color=blue, domain=0.3150:0.3950, samples=400]{17};
        \addplot[thick, samples=50, dashed, blue] coordinates {(0.3950,16)(0.3950,17)};
        \addplot[thick, color=blue, domain=0.3950:0.4150, samples=400]{16};
        \addplot[thick, samples=50, dashed, blue] coordinates {(0.4150,16)(0.4150,17)};
        \addplot[thick, color=blue, domain=0.4150:0.5450, samples=400]{17};
        \addplot[thick, samples=50, dashed, blue] coordinates {(0.5450,17)(0.5450,18)};
        \addplot[thick, color=blue, domain=0.5450:0.7500, samples=400]{18};
        \addplot[only marks, mark=*, color=blue]
        table{
            x  y
            0.7034 18
            };
    \end{axis}
\end{tikzpicture}
\caption{Function $\varepsilon \mapsto m(\varepsilon)$ for Example \ref{ex:different_tau}. We indicate as a blue dot the value of the function $m(\varepsilon^{\star})$, with $\varepsilon^\star = 0.7034$. We observe that Assumption \ref{assumption_points} holds, when $\varepsilon$ is sufficiently close to $\varepsilon^{\star}$.}
    \label{fig:Plot_points_m(eps)}
\end{figure}

The purpose of the outer iteration consists in solving the one dimensional root-finding problem $g(\varepsilon^{\star}) = 0$. In order to solve this task, we propose to tune the value $\varepsilon$ by using a combination of the Newton method and the bisection approach. A further generic assumption that we need in order to present a Newton-like iteration is the following:
\begin{assum}
\label{assumption}
The smallest singular value $\sigma_j(\varepsilon)$ of the matrix
\begin{equation*}
\sum_{i=1}^d f_i(\mu_j)\left( A_i + \varepsilon  \Delta_i (\varepsilon) \right)
\end{equation*}
is simple and different from $0$ and, moreover, $\sigma_j(\varepsilon), \boldsymbol{\Delta}(\varepsilon)$ are smooth with respect to $\varepsilon$, for each $j=1,\ldots,m$. 
\end{assum}
Note that if the Assumption was not fulfilled, the bisection method would still be applicable and allows us to compute $\varepsilon^*$. We propose here a method that combines the Newton's iteration and the bisection approach. Thanks to Assumption \ref{assumption}, we are able to compute the derivative of $g(\varepsilon)$ with respect to $\varepsilon$.

\begin{theorem}
\label{thm:derivativ_wrt_eps}
If the following properties hold:
\begin{itemize}
    \item[(i)] $\varepsilon \in \left( 0, \varepsilon^* \right)$, with $g(\varepsilon) >0$;
    \item[(ii)] Assumption \ref{assumption} holds true;
    \item[(iii)] $\mathbf{M}(\varepsilon) \neq 0$.
\end{itemize}
Then:
\begin{equation*}
    \frac{d}{d \varepsilon} g(\varepsilon)= - \| \mathbf{M}(\varepsilon) \|_F.
\end{equation*}
\end{theorem}

\begin{proof}
Thanks to Assumption \ref{assumption}, we are allowed to derive the functional with respect to $\varepsilon$:
\begin{equation}
\label{eq:derivate_G_epsilon}
    \frac{d}{d \varepsilon} g(\varepsilon) = \Big\langle \mathbf{M}(\varepsilon), \boldsymbol{\Delta}(\varepsilon) + \varepsilon \frac{d}{d\varepsilon}\boldsymbol{\Delta}(\varepsilon) \Big\rangle.
\end{equation}
Since, by definition, $\boldsymbol{\Delta}(\varepsilon)$ is a path of stationary points of \eqref{eq:gradient_system}, we may employ Theorem \ref{thm:charact_minim} and obtain $\boldsymbol{\Delta}(\varepsilon)= \pm \frac{\mathbf{M}(\varepsilon)}{\| \mathbf{M}(\varepsilon)\|_F}$. The constraint $\| \boldsymbol{\Delta}(\varepsilon) \|_F=1$ for all $\varepsilon$ leads to the relation:
\[
0 = \Big\langle  \boldsymbol{\Delta}(\varepsilon), \frac{d}{d\varepsilon}  \boldsymbol{\Delta}(\varepsilon) \Big\rangle = \pm \frac{1}{\| \mathbf{M}(\varepsilon) \|_F} \Big\langle \mathbf{M}(\varepsilon), \frac{d}{d \varepsilon} \boldsymbol{\Delta}(\varepsilon) \Big\rangle.
\]
Substituting this relation into \eqref{eq:derivate_G_epsilon}, we get $\frac{d}{d \varepsilon} g(\varepsilon) = \pm \| \mathbf{M}(\varepsilon) \|_F.$ Then, since using that, by assumption, $g(\varepsilon) >0$ for $\varepsilon \in ( 0 , \varepsilon^{\star})$ and that, by definition, $g(\varepsilon^{\star})=0$, we get that
\[
\frac{d}{d \varepsilon} g(\varepsilon) = - \| \mathbf{M}(\varepsilon) \|_F.
\]
\end{proof}
The assumption on the smoothness of the function $g(\varepsilon)$ with respect to $\varepsilon$ and the expression for the derivative of $g(\varepsilon)$ given by Theorem \ref{thm:derivativ_wrt_eps} allow us to apply Newton's method for $\varepsilon < \varepsilon^*$:
\begin{equation*}
    \varepsilon_{k+1}= \varepsilon_k - \left( \| \mathbf{M}(\varepsilon_k) \|_F\right)^{-1}  g(\varepsilon_k), \quad k=1,2,\ldots 
\end{equation*}
The quadratic convergence to the distance to singularity $\varepsilon^*$ is guaranteed only by the left and we need to include a bisection step to update the approximation by the right. If $g(\varepsilon_k) < \mbox{tol}$, for a chosen accuracy $\mbox{tol}$ , we interpret that $\varepsilon_k > \varepsilon^*$ and define
\begin{equation*}
    \varepsilon_{k+1}= \frac{ \varepsilon_k + \varepsilon_{k-1}}{2}.
\end{equation*}

\section{Extension to structured matrix-valued functions}
\label{sec:Extension to structured}

    The approach described in Subsections \ref{subsec:Inner for matrix-valued functions} and \ref{subsec:Outer for matrix-valued} can be extended to the case of structured matrix-valued functions, employing the definition of structured distance to singularity:
    \begin{definition}
    \label{def:structured_distance_matrix_functions}
        Consider a regular matrix-valued function $\mathcal{F}(\lambda)$ as in \eqref{eq:matrix-valued function} and a linear subspace $\mathcal{S} \subseteq \mathbb{C}^{n \times dn}$. Assume that $\left[A_d,\ldots,A_1 \right] \in \mathcal{S}$. Then, we define the \emph{structured distance to singularity} as
        \begin{equation*}
            d_{\rm sing}^{\mathcal{S}} \left( \mathcal{F}\right):= \min \left\lbrace \| \Delta \mathcal{F} \|:\det \left( \mathcal{F}(\lambda) + \Delta \mathcal{F}(\lambda)  \right)=0, \left[ \Delta A_d, \ldots, \Delta A_1 \right] \in \mathcal{S} \right\rbrace.
        \end{equation*}
    \end{definition}
A structured version of the optimization problem in \eqref{eq:opt_pbm_finite_n} can be obtained with the introduction of the additional constraint that is $\left[\Delta A_d,\ldots, \Delta A_1 \right] \in \mathcal{S}$. The extension to structured perturbations can be done introducing the notion of orthogonal projection onto the manifold of structured matrix-valued functions. In this setting, the orthogonal projection $\Pi_{\mathcal{S}}$ with respect to the real Frobenius inner product is a map from $\mathbb{C}^{n \times nd}$ to $\mathcal{S}$ such that for every $Z\in \mathbb{C}^{n \times nd}$
\begin{equation*}
    \Pi_{\mathcal{S}}\left( Z \right) \in \mathcal{S}, \quad \left\langle \Pi_{\mathcal{S}}\left( Z \right), W \right\rangle = \left\langle Z, W \right\rangle, \; \forall W \in \mathcal{S}.
\end{equation*}
Lemma \ref{lem:minim} can be flexibly specialized to structured perturbations in $\mathcal{S}$ and the resulting ODE system has the form
\begin{equation*}
    \dot{\boldsymbol{\Delta}}(t)= -\Pi_{\mathcal{S}}\left( \mathbf{M} \right) + \boldsymbol{\Delta} (t),
\end{equation*}
where $\mathbf{M}$ is defined as in Subsection \ref{subsec:Inner for matrix-valued functions}. Consequently, in this case, the stationary points of the functional \eqref{eq:functional_G} are negative multiples of the matrix $\Pi_{\mathcal{S}}\left( \mathbf{M} \right)$. In the outer iteration, the Newton step relies on the modified formula for the derivative:
\begin{equation*}
    \frac{d}{d \varepsilon} g(\varepsilon)=-\| \Pi_{\mathcal{S}} \left( \mathbf{M}(\varepsilon) \right)\|_F.
\end{equation*}
The generalizations of the Lemmas for the structured case and their proofs are analogous to the ones provided in Subsections \ref{subsec:Inner for matrix-valued functions} and \ref{subsec:Outer for matrix-valued}. An extensive description of the method applied to structured perturbations can be found in \cite{GnazzoGugl} for the case of matrix polynomials. The specialization of the method can be done for several structures. Specifically, it may include properties on individual coefficients, such as the presence of constraints for the matrix-valued function to be real. Another example is the one of considering a set of indices $I \subseteq \left\lbrace 1,\ldots, d \right\rbrace$, with cardinality $\left| I \right| < d$, then the distance to singularity with fixed coefficients $A_i$, if $i \in I$ can be computed minimizing the functional $G_{\varepsilon}\left( \boldsymbol{\Delta} \right)$ with the additional constraints ${\Delta}_i(t) \equiv 0$ for $i \in I$, which implies that in the ODE system in \eqref{eq:gradient_system} the indices $i \in I$  are excluded from the gradient system. We follow a similar approach when the matrices have individual properties, such as $\Delta_i \in \mathcal{S}_i$, with $i=1,\ldots,d$ and $\mathcal{S}_i \subseteq \mathbb{C}^{n \times n}$. This is, for example, the case of coefficients $A_i$ with the additional constraint of sparsity patterns. In this setting, the projection $\Pi_{\mathcal{S}}$ can be split using the orthogonal projections $\Pi_{\mathcal{S}_i}$ on each coefficient $A_i$ and the map $\Pi_{\mathcal{S}_i}$ coincides with the identity for the entries in the sparsity pattern and with the null function on the remaining ones:
\begin{align*}
   \left\lbrace \Pi_{\mathcal{S}_i} \left( B \right) \right\rbrace_{j,l}:= \left\lbrace \begin{array}{c}
        B_{j,l}, \quad \mbox{if} \; (j,l) \not\in \mathcal{T}\\[2mm]
         0, \quad \mbox{if} \; (j,l) \in \mathcal{T},
   \end{array} \right. ,
\end{align*}
where $\mathcal{S}_i:= \left\lbrace B \in \mathbb{C}^{n \times n} : B_{j,l}=0 \; \mbox{if} \; (j,l) \in \mathcal{T} \right\rbrace$, for a certain subset of ordered pairs $\mathcal{T}$ of $\left\lbrace 1,\ldots,n \right\rbrace \times \left\lbrace 1,\ldots,n \right\rbrace$. Moreover, we can also adapt the method to matrix-valued functions with collective-like properties, for instance palindromic functions. Note that the computation of the orthogonal projection can be more difficult in these situations, due to the relations among the coefficients of the matrix-valued function.

\section{A novel perspective for the case of matrix polynomials}
\label{sec:comparison_matrix_pol}

The class of matrix-valued functions $\mathcal{F}(\lambda)$ in \eqref{eq:matrix-valued function} contains, of course, the set of matrix polynomials: given $d \in \mathbb{N}$, we may consider the set of matrix polynomials of degree lower or equal to $d-1$,
\begin{equation*}
    \mathcal{P}(\lambda)=\lambda^{d-1} A_d + \ldots + \lambda A_2 + A_1,
\end{equation*}
where $A_i \in \mathbb{C}^{n \times n}$. In \cite{GnazzoGugl} we proposed a suitable method for the approximation of the distance to singularity for matrix polynomials, which relies on the fundamental theorem of algebra. There we select a number $\widetilde{m}:=(d-1)n +1$ of distinct complex points $\left\lbrace \mu_j \right\rbrace$, which is the largest possible degree of the determinant of $\mathcal{P}(\lambda)$ plus $1$. The number of chosen points may be modified in this new perspective, using the criterion for the selection of the number $m$ provided in Section \ref{sec:problem setting}. 

The selection of the number of points relying on the proposed criterion may significantly reduce $m$ with respect to $\widetilde{m}$. This may be of particular importance in situations where the matrix coefficients present a certain sparsity pattern and we are interested in computing the structured distance to singularity with respect to the same sparsity pattern. Indeed, in this context a sparse leading coefficient matrix may produce a determinant of degree lower that $\widetilde{m}$. However, our novel approach may provide a number $m$  even smaller than $\widetilde{m}$. We provide a few numerical tests of this advantage in Subsection \ref{subsec:special case_matrixpol}, comparing the results obtained with this novel approach and the ones obtained with the fundamental theorem of algebra. 

In order to provide a fair comparison between the two methods, we employ a slightly different functional $G_{\varepsilon}$, with respect to the one provided in \eqref{eq:functional_G}. Indeed, we consider the scaled functional
\begin{equation}
\label{eq:scaled_functional_widetildeG}
    \widetilde{G}_{\varepsilon}\left( \Delta_d,\ldots,\Delta_1 \right)= \frac{1}{2m^2} \sum_{j=1}^m \sigma_j^2 \left( \Delta_d, \ldots, \Delta_1 \right).
\end{equation}
All the theorems stated in Subsection \ref{subsec:Inner for matrix-valued functions} remain true for the modified functional $\widetilde{G}_{\varepsilon}$. Remarkably, our method can be applied to matrix polynomials with additional structures, as described in Section \ref{sec:Extension to structured}. Since this is possible also with the approach proposed in \cite{GnazzoGugl}, we compare the two methods in the structured case. The numerical experiments in Subsection \ref{subsec:special case_matrixpol} show that this novel approach leads to a speed-up in time, yielding to an improvement of the state-of-the-art for the structured case. To our knowledge, other existing methods \cite{DasBora,GiesHaral} are currently not implemented for dealing with the structured case. Nevertheless, we may still test the behaviour of our approach in the unstructured case, comparing with the methods in \cite{ByersHeMehr,DasBora,GiesHaral}. We illustrate this comparison in Subsection \ref{subsec:special case_matrixpol}.

\section{Computational issues}
\label{sec:Computational issues}

The computational aspects of the method require a more careful discussion. Firstly, the numerical implementation of the method may require a initial normalization. Indeed, it may be useful to perform a normalization if the determinant of the initial matrix-valued function has a large or also small modulus. This is equivalent to work with a relative error. The normalization we propose is the following. Consider a discrete set of points $\Xi \subseteq \partial D$. As explained in Section \ref{subsec:Reformulation}, we may consider as $D$ the unit disk in the complex plane. Therefore, we select a finite number of distinct points in $\partial D$:
\[
\Xi := \left\lbrace z_j =e^{\frac{2\pi i }{p}j}, \; j=1,\ldots,p \right\rbrace.
\]
We scale the coefficient matrices $A_i$ for $i=1,\ldots,d$ dividing them by a certain factor $\alpha$ and we obtain new coefficients $\widetilde{A}_i$. The normalized matrices $\widetilde{A}_i$, $i=1,\ldots,d$ are taken in such a way that
\begin{equation*}
   \max_{j=1,\ldots,p} \left| \det \left( \sum_{i=1}^d f_i(z_j) \widetilde{A}_i \right) \right|=1.
\end{equation*}
The effective distance to singularity for $\mathcal{F}$  is taken multiplying the distance to singularity for $\widetilde{\mathcal{F}}$ by the factor $\alpha$. In our implementation of the approach, this normalization is obtained fixing a prescribed number of points $p$ and evaluating the determinant of $\sum_{i=1}^d f_i(z_j) A_i$ at the points $z_j \in \Xi$. Even if other choices of $\Xi$ could lead to a different normalization, we did not observe relevant changes in our numerical experiments.

As described in Subsections \ref{subsec:Inner for matrix-valued functions} and \ref{subsec:Outer for matrix-valued}, our approach relies on a nested optimization procedure. In the inner iteration, we fix the value of the perturbation $\varepsilon$ and optimize the functional $G_{\varepsilon}(\boldsymbol{\Delta})$  in \eqref{eq:functional_G}, over admissible $\boldsymbol{\Delta}$. In the outer iteration, we tune of $\varepsilon$ using a Newton-bisection method. Between these two iterations, the inner one is more delicate. Then, it is useful to add a few details about it. In particular, we seek the stationary points of the constrained gradient system \eqref{eq:gradient_system}, which are the local extremizers of the functional $G_{\varepsilon}$, as proved in Theorem \ref{thm:charact_minim}. To tackle this problem, we employ an explicit Euler method for the matrix ODE system and normalize the obtained perturbation $\boldsymbol{\Delta}$ at each step, in order to have unit Frobenius norm. It is often convenient to apply an adaptive choice of the step-size, for instance employing an Armijo-type line-search \cite{Luenberger}.
 
 The ODE technique may be particularly convenient in situations where $m \ll n$. Indeed, in this case, it can be proved that the matrix coefficients of the perturbation $\Delta \mathcal{F}$ in \ref{eq:det_at_points_section_numerical_sing_matrices} are rank $m$ matrices, see for instance the idea in \cite{Robol}. Then, it is possible to work separately on the factors of the low-rank matrices, by associating two differential equations, as proposed in \cite{GuglLubich}. A similar approach may be also employed when working with structured matrix-valued functions, following the idea in \cite{Sicilia}. Nevertheless, this optimization procedure is not the only possible choice for the inner iteration. For instance, higher order methods are available: one example is the Matlab toolbox \texttt{manopt}, available at \url{https://www.manopt.org/}. We tested this alternative proposal in our numerical tests, and we refer the reader to Subsection \ref{subsec:further_implementation_strategies} for a detailed comparison.

\subsection{Experimental choice of the number of needed points}
\label{subsec:experimental_points}

As we have explained in Section \ref{sec:problem setting}, the choice of the number $m$ and of the set of points $\left\lbrace \mu_j \right\rbrace_{j=1}^m$ represents a delicate step of the method. Although the theoretical background presented in Section \ref{sec:Two-level approach} relies on a fixed value of $m$, in our numerical implementation we opt for a successive adaptation of the number of needed points. This approach is useful to optimally approximate the function $\det\left( \mathcal{F}(\lambda) + \Delta \mathcal{F}(\lambda) \right)$. We propose an experimental approach based on the results on analytic functions provided by Theorems \ref{thm:Thref_approx_pol} and \ref{thm:Trapezoidal_Trefethen}. Our idea consists in a progressive adaptation of the number of points $m$ needed for a sufficiently accurate approximation of the analytic function $\det\left( \mathcal{F}(\lambda) + \Delta \mathcal{F}(\lambda) \right)$.

At each step of the approach, we have a function 
\begin{equation*}
\label{eq:function_f_k}
    \mathbf{F}_k(\lambda):=\det \left( \mathcal{F}(\lambda) + \Delta \mathcal{F}_k ( \lambda ) \right),
\end{equation*}
given by the computation of $\Delta \mathcal{F}_k(\lambda):=\sum_{i=1}^d f_i(\lambda) \Delta A_{i,k}$ at the $k$-th step of the outer iteration. For simplicity, we denote by $\mathbf{F}_0(\lambda)$ the determinant of $\mathcal{F}(\lambda)$. We fix a certain tolerance ${\rm{tol}}$ and we start the method choosing the smallest value of $m$ for which we have that the $m$-th coefficient of the Taylor expansion is 
\begin{equation}
\label{eq:condition_trap_rule_0}
    \left| a_{m}\left( \mathbf{F}_0(\lambda)\right) \right| \leq {\rm tol}.
\end{equation}
More in detail, we propose to approximate numerically the quantity $a_m \left( \mathbf{F}_0(\lambda) \right)$, by using the trapezoidal rule with points $e^{\frac{2\pi i}{N}j}$, as in Theorem \ref{thm:Trapezoidal_Trefethen}, using a sufficiently high number $N$. The smallest number of points that satisfies the condition \eqref{eq:condition_trap_rule_0}, denoted by $m_0$, will be the number of initial points and consequently
\begin{equation*}
    e^{\frac{2\pi i}{m_0}j}, \qquad j=1,\ldots,m_0,
\end{equation*}
will be the starting set of distinct points. Then, the idea consists in being able to provide a reliable approximation of the determinant $\mathbf{F}_k(\lambda)$, at each step of our method. This means being able to compute an accurate polynomial approximation $p(\lambda)$, using interpolant polynomials at the roots of the unit, as suggested by Theorem \ref{thm:Thref_approx_pol}. In this setting, it may happen that, at step $k$, by applying Theorem \ref{thm:Thref_approx_pol}, the function $\mathbf{F}_k(\lambda)$ requires an higher or lower number of points $m_k$, with respect to the one needed at the previous step. Therefore, we approximate numerically the Taylor coefficients of $\mathbf{F}_k(\lambda)$ and choose the smallest $\widetilde{m}$ such that
\begin{equation}
\label{eq:update_mtilde}
        \left| a_{\widetilde{m}}\left( \mathbf{F}_k(\lambda)\right) \right| \leq {\rm tol}.
\end{equation}
Then, as for the first iteration of the method, we could choose $m_k:= \widetilde{m}$ and consider the new set of points
\begin{equation*}
    e^{\frac{2\pi i}{m_k}j}, \qquad j=1,\ldots,m_k.
\end{equation*}
However, it would be convenient to avoid the re-computation of the value $m_k$ at each step of the iterative method. Indeed, if the coefficients of the functions $\Delta \mathcal{F}_k(\lambda)$ and $\Delta \mathcal{F}_{k-1}(\lambda)$ do not vary much, it would be reasonable to keep the amount of points $m_{k-1}$ for the iteration $k$. To this aim, we suggest a strategy able to detect a sudden change in the perturbation functions $\Delta \mathcal{F}_k$. We propose to monitor the behaviour of the method by adding a control on the relative increase of the norm of the perturbation matrices $\mathcal{A}_k:=\left[A_d+ \Delta A_{d,k} \, \ldots, A_1+ \Delta A_{1,k} \right]$. In particular, at each step of the outer iteration, we compute the quantity
\begin{equation}
\label{eq:ratio_norm}
  \mathcal{R}_k:=\frac{\| \mathcal{A}_k - \mathcal{A}_{k-1} \|_F}{\| \mathcal{A}_{k-1} \|_F},
\end{equation}
and if the ratio \eqref{eq:ratio_norm} $\mathcal{R}_k > \widetilde{\mbox{tol}}$ (in our experiments $\widetilde{\mbox{tol}}=0.001$ appears to be the most reliable), we perform a new computation of the number of needed points and update $m$. This additional control may avoid an high number of re-computations of the number of the needed points, allowing us to skip a few estimations of the value $m_k$, as illustrated in \eqref{eq:update_mtilde}.

\begin{remark}
    It is worth noticing that the design of our iterative method often leads to a different behaviour in the first steps and in the last ones. In details, we recall that we are employing a Newton-like method in order to solve the univariate root-finding problem, when we approach $\varepsilon^{\star}$ by the left. This may reflect to small changes of the computed variable $\varepsilon_k$ in the last iterations of the method, once we are converging to the solution $\varepsilon^{\star}$. For instance, one example of this possibility can be found in \cite[Table 7.2]{Sicilia}, where, in particular, the  values $\varepsilon_k$ vary in the order of $10^{-5}$ in the final steps. In our setting, a small variation in the value $\varepsilon$ leads to a small variation in the coefficients of the determinant of the matrix-valued function with coefficients $A_i + \varepsilon \Delta_i(\varepsilon)$, for $i=1,\ldots,d$. Then, we expect that at the very last iterations, it will not be necessary to change the number of needed points and Assumption \ref{assumption_points} is satisfied. This is also supported by our numerical experiments. Moreover, since our approach works with numerically singular matrices, we do not experience an exact singularity of the matrix-valued function. This means that the last steps of our method do not provide a number of points $m$ equal to $0$, since we do not reach the exact solution, but we only provide its numerical approximation.
\end{remark}

We now focus on strategies for the update of the number of needed points $m$ throughout the algorithm. Indeed, based on the relative increase $\mathcal{R}_k$ in \eqref{eq:ratio_norm}, we may need to repeat the computation of the number of points. To this end, we discuss two possible approaches for the update of the number of needed points $m_k$ at the $k$-th step of the procedure, where $\mathcal{R}_k > \widetilde{\mbox{tol}}$. A first course of action could be computing the new expected number of points $\widetilde{m}$, as in \eqref{eq:update_mtilde}, and set $m_k := \widetilde{m}$. This strategy would lead to a very accurate proposal, since, at almost each step, we would consider the exact number of distinct complex numbers $\mu_i$ that we need, as suggested by Theorem \ref{thm:Thref_approx_pol}. In our experience, moreover, this possibility may also lead to a speed-up in the numerical implementation of the method. We refer to this proposal as \emph{Subsequent} update of $m$.

Nevertheless, we also provide a second possibility, partially supported by the theory. Indeed, as proved in Theorem \ref{thm:derivativ_wrt_eps}, we have a monotonicity property for the functional $G_{\varepsilon}$. Therefore, it would be useful to keep this property into our numerical implementation of the method. To this end, we propose the following experimental strategy. If the new number $\widetilde{m}$ of points needed for a sufficiently accurate approximation of the determinant is greater than the double or smaller than the half of the previous one, that is $m_{k-1}$, we consider the new suggested set of points, increasing or decreasing its cardinality accordingly. This approach could lead to an increase of the computation cost, but it provides a safer strategy for the update of the number of points $m$. To explain this last statement, we separate two case:
\begin{enumerate}
    \item[\emph{(i)}] $\widetilde{m}  \leq \frac{1}{2} m_{k-1}$. In this case, we choose $m_k:= \frac{1}{2} m_{k-1}$;
    \item[\emph{(ii)}] $\widetilde{m} \geq 2 m_{k-1}$. In this case, we choose $m_k:= 2 m_{k-1}$.
\end{enumerate}
In the situation $\emph{(i)}$, at the end of the $(k-1)$-th step, we have that:
\[
G_{\varepsilon,  m_{k-1}} ( \boldsymbol{\Delta}(\varepsilon)) = \frac{1}{2} \sum_{j=1}^{m_{k-1}} \sigma_j^{2} (\boldsymbol{\Delta}(\varepsilon)) > \frac{1}{2} \sum_{j=1}^{m_{k}} \sigma_j^{2} (\boldsymbol{\Delta}(\varepsilon)) = G_{\varepsilon, m_{k}}  ( \boldsymbol{\Delta}(\varepsilon)),
\]
for each $\varepsilon $ in the interval $(\varepsilon_{k-1},\varepsilon^{\star})$. This procedure avoids possible increases of the functional $G_{\varepsilon}$ and this could be useful when we are approaching (local) extremizers of the functional.

In the situation \emph{(ii)}, we consider the case where the value $\varepsilon^{\star}$  has not been reached yet, that is $G(\varepsilon_{k-1}) > \omega$, with $\omega >0$. Then, for values of $\varepsilon$ in $(\varepsilon_{k-1}, \varepsilon^{\star})$, the following relation holds:
\[
\omega < G_{\varepsilon, m_{k-1}}(\boldsymbol{\Delta} (\varepsilon)) = \frac{1}{2} \sum_{j=1}^{m_{k-1}} \sigma_j^2(\boldsymbol{\Delta}(\varepsilon)) <  \frac{1}{2} \sum_{j=1}^{m_{k}} \sigma_j^2(\boldsymbol{\Delta}(\varepsilon))  = G_{\varepsilon, m_{k}} (\boldsymbol{\Delta}(\varepsilon)).
\]
This allows us to restrict the study of the functional $G_{\varepsilon,m_k}(\varepsilon)$ in the interval $(\varepsilon_{k-1},\varepsilon^{\star})$. We refer to this second proposal as \emph{Half-Double} update of $m$. In the numerical implementation of our method, we employ this second possibility.

Note that, in principle, for the \emph{Half-Double} technique, we could choose a different update. To illustrate this, let us consider case \emph{(i)}. Indeed, as long as $m_k \leq \frac{1}{2} m_{k-1}$, the monotonicity property still holds. Then, one could choose any subset $\mathcal{I} \subseteq \lbrace e^{\frac{2\pi i}{m_{k-1}}j} \rbrace_{j=1}^{m_{k-1}}$ of cardinality $m_k$. However, this approach would not guarantee the bound on the approximation of the determinant proposed in Theorem \ref{thm:Thref_approx_pol}, which requires equally spaced points on the unit disk.

Note that these proposed approaches are experimental strategies. We experimentally observed that the performances of the two techniques are comparable, in terms of the approximated distance to singularity and of the accuracy of the produced solution. We provide a comparison of the two ideas in Subsection \ref{subsec:further_implementation_strategies}.

In the following, in Algorithm \ref{Alg:choice of m}, we provide a pseudocode suggesting how to implement the choice of the number of points. Note that Lines from $8$ to $13$ provide an additional robustness check. We describe this procedure in the subsequent Subsection \ref{subsec:robustness_theory}. Finally, we provide a pseudocode for the iterative method in Algorithm \ref{Alg:distance to sing functions}. Our implementation of the method in Matlab is freely available at \url{https://github.com/miryamgnazzo/nearest-singular-matrix-valued-function}. We refer the reader to the codes for more implementation details and precise choices for the employed parameters. 

\algblock{Begin}{End}
\begin{algorithm}[!h]
\caption{Choice of number of points}
\label{Alg:choice of m}
\hspace*{\algorithmicindent} \textbf{Input}: scalar function $f(\lambda)$, minimum number of points $m_{\min}$, maximum number of points $m_{\max}$, tolerance $\mbox{tol}$
\\
\hspace*{\algorithmicindent} \textbf{Output}: number of needed points $\widetilde{m}$
\begin{algorithmic}[1]
\Begin
\State Set $m=m_{\min}$
\State Set $\left|a_m\right|= 2 \mbox{tol}$
\While{$\left|a_m \right| \geq \mbox{tol}$ and $m \leq m_{\max}$}
\State Approximate $a_m$ in \eqref{eq:coeff_an} for the function $f(\lambda)$, using trapezoidal rule 
\State $m=m+1$
\EndWhile
\State (Robustness check) Set $b_m$ equal to \eqref{eq:mean} 
\If{ $ b_m \leq \mbox{tol}$ }
\State Set $\widetilde{m}:=m$
\Else $\; m=m+1$
\State Repeat from Line $3$
\EndIf
\End
\end{algorithmic}
\end{algorithm}

\algblock{Begin}{End}
\begin{algorithm}[!h]
\caption{Approximation of $d_{{\rm{sing}}}$}
\label{Alg:distance to sing functions}
\hspace*{\algorithmicindent} \textbf{Input}: Matrices $A_d,\ldots,A_1$, tolerances $\mbox{tol}_i$ for $i=1,2,3$, threshold $\beta$, initial values $\varepsilon_0$, $\varepsilon_{\mbox{low}}$, $\varepsilon_{\mbox{up}}$\\ 
\hspace*{\algorithmicindent}\textbf{Output}: Upper bound for the distance $\varepsilon^*$, perturbation matrices $\Delta_d^*,\ldots, \Delta_1^*$
\begin{algorithmic}[1]
\Begin
\State Apply Algorithm \ref{Alg:choice of m} to $f^{0}\left( \lambda \right)$, with tolerance $\mbox{tol}_3$
\State $m_0:=\widetilde{m}$ from Algorithm \ref{Alg:choice of m}
\State Construct the set $e^{\frac{2 \pi i }{m_0}j}$, for $j=1,\ldots,m_0$
\State Initialize $m_{\rm old}=m_0$
\State Initialize $\Delta_d (\varepsilon_0), \ldots, \Delta_1(\varepsilon_0)$
\State Compute $g(\varepsilon_0), g'(\varepsilon_0)$
\State Set $k=0$
\While{$k \leq k_{\max}$ or $\left| \varepsilon_{\rm up}- \varepsilon_{\rm low} \right| > {\rm tol}_2$}
\If {$g \left( \varepsilon_k \right) > {\rm tol}_1$} 
\State $\varepsilon_{\rm low}=\max \left( \varepsilon_{\rm low}, \varepsilon_k \right)$
 \State Compute $\varepsilon_{k+1}$ using a Newton step
\Else
\State $\varepsilon_{\rm up}= \min \left( \varepsilon_{\rm up}, \varepsilon_k \right)$
 \State $\varepsilon_{k+1}=\left( \varepsilon_{\rm low} + \varepsilon_{\rm up} \right)/2$
\EndIf
\State Set $k=k+1$
 \If {$\varepsilon_k \not\in \left[ \varepsilon_{\rm low}, \varepsilon_{\rm up} \right]$}
 \State $\varepsilon_k= \left( \varepsilon_{\rm low} + \varepsilon_{\rm up} \right)/2$ 
 \EndIf
\State Compute $\boldsymbol{\Delta} \left( \varepsilon_k \right)$ solving the ODE \eqref{eq:gradient_system}
\State Compute $g\left( \varepsilon_k \right)$
\State Compute the ratio $\mathcal{R}_k$ \eqref{eq:ratio_norm}
\If{$\mathcal{R}_k \geq \beta$}
\State Apply Algorithm \ref{Alg:choice of m} to $f^{k}\left(\lambda \right)$, with tolerance $\mbox{tol}_3$
\State $m_{\rm new}:=\widetilde{m}$ from Algorithm \ref{Alg:choice of m}
\If{$m_{\rm  new} \geq 2 m_{\rm old}$}
\State Set $m_{\rm new}:=2m_{\rm old}$
\Else \If{$m_{\rm new} \leq m_{\rm old}/2$}
\State Set $m_{\rm new}:=\lfloor \frac{m_{\rm old}}{2}\rfloor$
\EndIf
\EndIf
\State Construct the set $e^{\frac{2 \pi i }{m_{\rm new}}j}$, for $j=1,\ldots,m_{new}$
\State Set $m_{\rm old}=m_{\rm new}$
\EndIf
\EndWhile
 \State Set $\varepsilon^*= \varepsilon_k$
 \State Set $\boldsymbol{\Delta}^*=\boldsymbol{\Delta} \left( \varepsilon_k \right)$
\End 
\end{algorithmic}
\end{algorithm}

\subsection{Robustness of the numerical method}
\label{subsec:robustness_theory}
The main point of the approach we propose consists in the possibility of substituting the condition on the determinant $\det \left( \mathcal{F}(\lambda) + \Delta \mathcal{F}(\lambda) \right)$ by a discrete condition on the set of complex numbers $\lambda = \mu_j$, for $j=1,\ldots,m$. One important issue to analyze is the stopping criterion in choosing the number of needed points $m$ in Algorithm \ref{Alg:choice of m}. Indeed, the choice performed in Algorithm \ref{Alg:choice of m} involves the numerical approximation of the $m$-th coefficient of the Taylor series $a_m$. Thanks to Theorem  \ref{thm:Thref_approx_pol}, we get that the interpolation error goes to zero, decaying as $\rho^{-m}$, for some $\rho>1$, for $m\to \infty$. However, even if we expect a fast decay of the values $\left| a_m\right|$ since we are considering entire functions, it is appropriate to include an additional procedure to check the robustness of the process. In detail, we suggest checking that the interpolant polynomial $p(\lambda)$ at the points $\mu_j = e^{\frac{2\pi i}{m}j}$, for $j=1,\ldots,m$ is indeed an appropriate approximation of the function $f(\lambda) = \det \left(\mathcal{F}(\lambda)+ \Delta \mathcal{F}(\lambda)\right)$. In Algorithm \ref{Alg:choice of m}, Line $8$, we consider the computed value $m$, and compute the coefficients of the polynomial $p(\lambda)$ of degree $m-1$, such that
\[
p(\mu_j) = f(\mu_j), \; \mu_j = e^{\frac{2\pi i}{m}j},\mbox{ for } j=1,\ldots,m.
\]
This step could be implemented in a relatively cheap way, since the evaluations $f(\mu_j)$ have already been computed in the previous steps of the approach, and the derivation of the coefficients of $p(\lambda)$ can be done employing the FFT, with a computational cost of $\mathcal{O}(m \log m)$ \cite{CooleyTukey}. Afterward, we consider a discrete set of points $\Omega \subseteq D$ (that does not include the points $\left\lbrace \mu_j \right\rbrace_{j=1}^m$), where $D$ is the unit disk in the complex plane, and we evaluate the quantity
\begin{equation}
  b_m := \max_{\omega \in \Omega} \left|p(\omega) - f(\omega) \right|.
    \label{eq:mean}
\end{equation}
If this value is smaller than the tolerance tol prescribed in Algorithm \ref{Alg:choice of m}, we can accept it, otherwise, we increase $m$, and repeat the whole procedure.

The described technique can be included in our novel approach: the dominant part of the computational cost does not vary. In Subsection \ref{subsec:robustness}, we provide numerical experiments, considering this check on the robustness of the criterion.

\section{Numerical examples}
\label{sec:Numerical examples matrix valued}

\begin{ex}
\label{ex:time delay nlevp}
We consider the characteristic equation of a time-delay system with a single delay and constant coefficients in the form
\begin{equation*}
    \mathcal{F}_1(\lambda)=-\lambda A_2+ \exp{(-\lambda)} A_1 + A_0,
\end{equation*}
where we have a $3 \times 3$ matrix-valued function with coefficients
\begin{equation}
\label{eq:time delay coeff_ nlevp}
    A_1=\left[ \begin{array}{c c c}
       0  & 1 & 0 \\
        0 & 0 & 1 \\
       -a_0 & -a_1 & -a_2 \\ 
    \end{array} \right], \quad   A_0=\left[ \begin{array}{c c c}
       0  & 0 & 0 \\
        0 & 0 & 0 \\
       -b_0 & -b_1 & -b_2 \\ 
    \end{array} \right],
\end{equation}
and $A_2:=I_{3}$ is the identity matrix of size $3 \times 3$. This example is taken from \cite{JarMich} and it is part of the \texttt{nlevp} collection of nonlinear eigenvalue problems proposed in \cite{NLEVP} (problem \texttt{time\_delay}, page $19$). As explained in Section \ref{sec:Computational issues}, we apply the method to the normalized version $\widehat{\mathcal{F}}_1(\lambda)$ of the matrix-valued function $\mathcal{F}_1(\lambda)$. The resulting matrix-valued function is the approximation of the closest singular matrix-valued function for $\widehat{\mathcal{F}}_1(\lambda)$.

We start computing the distance to singularity for $\widehat{\mathcal{F}}_1(\lambda)$ without including the additional information about the sparsity pattern. Nevertheless, since the matrices involved are real, it is reasonable to ask for real perturbations in $\Delta \widehat{\mathcal{F}}_1(\lambda)$. Then we apply the method adding the constraint of real perturbations. Setting for the tolerance $\rm{tol}$ in Algorithm \ref{Alg:choice of m} equal to $10^{-12}$, the initial number of points is $15$ and in this case we do not need to recompute them. We set ${\rm tol}_1=15\times 10^{-8}$ and ${\rm tol}_2 =10^{-6}$ in Algorithm \ref{Alg:distance to sing functions}. The approximated distance to singularity with real perturbations is equal to $0.0450$ (with $4$ digits).

We can check the quality of the singularity at a dense set of points: consider the points $\left\lbrace x_j+i y_j \right\rbrace$ given by \texttt{[X,Y] = meshgrid(-0.9:0.01:0.9)}, we evaluate the smallest singular value of $\widehat{\mathcal{F}}_1(\lambda)+ \Delta \widehat{\mathcal{F}}_1(\lambda)$ at these points. In this case, the maximum value is $ 4.1935 \times 10^{-3}$ and a minimum value is $7.0389 \times 10^{-7}$. In Figure \ref{fig:contour1}, we plot the evaluations of the smallest singular value of $\widehat{\mathcal{F}}_1(\lambda)+ \Delta \widehat{\mathcal{F}}_1(\lambda)$ at the points obtained using \texttt{meshgrid}.

\begin{figure}[h!]
    \centering
    \includegraphics[width=0.65\textwidth]{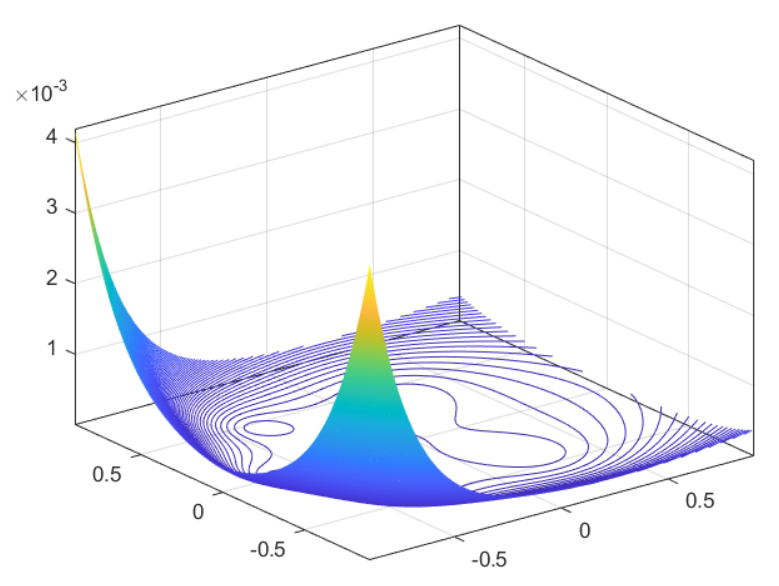}
    \caption{Contour plot for the smallest singular value of $\widehat{\mathcal{F}}_1(\lambda)+ \Delta \widehat{\mathcal{F}}_1(\lambda)$ at the grid of points obtained with {\tt meshgrid}.}
    \label{fig:contour1}
\end{figure}

In this example the final perturbations $\Delta A_2, \Delta A_1, \Delta A_0$ are full matrices and do not respect the initial sparsity pattern of the coefficients. Indeed, if we want to consider the structure induced by the sparsity pattern, we can extend the method as in Section \ref{sec:Extension to structured}, with the introduction of an additional constraint on the optimization problem. In this setting, for instance, the sparsity pattern induced by the matrices $A_2,A_1,A_0$, respectively, is
\begin{equation*}
    \left[ \begin{array}{c c c}
        * & 0 & 0 \\
        0 & * & 0 \\
        0 & 0 & * \\
    \end{array} \right],  \quad \left[ \begin{array}{c c c}
        0 & * & 0 \\
        0 & 0 & * \\
        * & * & * \\
    \end{array} \right], \quad  \left[ \begin{array}{c c c}
        0 & 0 & 0 \\
        0 & 0 & 0 \\
        * & * & * \\
    \end{array} \right].
\end{equation*}
Here we use the same tolerances and the same number of points $m=15$ considered for the case without sparsity constraints and the method does not need a recomputation of the number of points. The structured distance to singularity (up to $4$ digits) is $0.0833$. In an analogous way, we produce an evaluation of the quality of the singularity, which provides a maximum of the smallest singular value of $\widehat{\mathcal{F}}_1(\lambda)+ \Delta \widehat{\mathcal{F}}_1(\lambda)$ of $1.4868 \times 10^{-3}$ and a minimum value of $3.4674 \times 10^{-7}$, using the grid provided before.

A third possibility may be preserve the structure associated with the time-delay equation. In particular, we may ask for perturbations in the form
\begin{equation*}
   \Delta A_2= \left[ \begin{array}{c c c}
        0 & 0 & 0 \\
        0 & 0 & 0 \\
        0 & 0 & 0 \\
    \end{array} \right],  \quad \Delta A_1= \left[ \begin{array}{c c c}
        0 & 0 & 0 \\
        0 & 0 & 0 \\
        * & * & * \\
    \end{array} \right], \quad  \Delta A_0= \left[ \begin{array}{c c c}
        0 & 0 & 0 \\
        0 & 0 & 0 \\
        * & * & * \\
    \end{array} \right].
\end{equation*}
This choice preserves the $0-1$ structure arising in the time delay equation and only perturbs the coefficients $a_i, b_i$ for $i=0,1,2$ in \eqref{eq:time delay coeff_ nlevp}. In this setting, the problem seems very far from a singular matrix-valued function since its approximated distance to singularity is $84.47756$. The evaluations at the points provided by \texttt{meshgrid} has maximum value equal to $4.2618 \times 10^{-4}$ and minimum value of $3.1247 \times 10^{-7}$. Setting the threshold $\beta=10^{-3}$, the method often recomputes the number of needed points, but in the end the number of needed points remains equal to $m=15$ for the whole computation.

\end{ex}

\begin{ex}
\label{ex:different_tau}
Consider the following $2 \times 2$ randomly generated matrices in MATLAB:
\texttt{n=2;rng(2); A$_2$=rand(n); A$_1$=randn(n); A$_0$=randn(n)}.
We construct a list of examples with different delays, in the following way:
    \begin{equation*}
        \mathcal{F}_{\tau} \left( \lambda \right)= \lambda A_2  + e^{-\tau \lambda} A_1 + A_0,
    \end{equation*}
    and test the method with several different values of $\tau$. As in Example \ref{ex:time delay nlevp}, the matrix-valued function $\mathcal{F}_{\tau}$ are normalized in order to have determinant equal to one on the boundary of the complex unit disk. For every run of the method, we ask for a precision of $\mbox{tol}=10^{-12}$ in the approximation of the coefficient of the Taylor series in Algorithm \ref{Alg:choice of m}. The final precision required for the approximation of the functional $G_{\varepsilon}$ is $\rm{tol}_1=m \cdot 10^{-10}$ and $\rm{tol}_2=10^{-6}$, with $\rm{tol}_1, \rm{tol}_2$ defined in Algorithm \ref{Alg:distance to sing functions}. Note that in this and all the other numerical examples, we choose as tolerance $\mbox{tol}_1$ a multiple of the number of points $m$. This is done because, in the numerical implementation of the method, this choice allows us to change the accuracy that we require on the functional $g(\varepsilon)$ (that depends on $m$), as we increase or decrease the number $m$ of needed points. In Table \ref{tab:table_multilple_delays}, for different values of $\tau$, we provide the approximated distance to singularity, the maximum number of needed points $m$, the maximum and the minimum value of the evaluations of the smallest singular value of the perturbed matrix-valued function at the set of points obtained with \texttt{meshgrid} and the number of iterations needed for reaching the required accuracy.

\begin{table}[h!]
   \caption{Results for $\mathcal{F}_{\tau}$ with different values of $\tau$.}
    \centering
    \begin{tabular}{|c|c|c|c|c|c|}
    \hline
       $\tau$  & Distance &N. points & Max $\sigma_{min}$ & Min $\sigma_{min}$ & Iter. \\
       \hline \hline
       $0.5$ & $0.8012$ & $15$ & $ 1.4032 \times 10^{-5}$ & $1.7917 \times 10^{-8}$ & $26$ \\
       $1$ & $0.7034$ & $19$ & $1.6654 \times 10^{-5}$ & $1.7759 \times 10^{-10}$ & $19$ \\
        $2$  & $0.4701$ &$26$  & $2.6116 \times 10^{-5}$ &  $5.2978 \times 10^{-8}$ & $16$ \\
        $3$ & $0.1208$ & $31$ & $4.2658 \times 10^{-5}$ & $8.1547\times 10^{-8}$ & $15$\\
      \hline
    \end{tabular}
    \label{tab:table_multilple_delays}
\end{table}    

\end{ex}

\begin{ex}
We now consider the following matrix-valued function:
\begin{equation*}
\mathcal{A}(\lambda)=\lambda A_2 + e^{-\tau_1 \lambda} A_1 + e^{\tau_2 \lambda} A_0, \quad \tau_i >0, \; i=1,2.
\end{equation*}
This kind of matrix-valued function may arise dealing with systems of advanced retarded differential equations, also known as mixed-type functional differential equations, see for instance \cite{Myshkis}. The study of this class of equations can be found in the context of quantum mechanics, as it is shown in \cite{Rodriguez} for quantum photonic circuits.
We can compute the distance to singularity for the matrix-valued function $\mathcal{A}(\lambda)$. In this example, we choose $\tau_1=\tau_2=1$ and we considered three random matrices of size $3 \times 3$, randomly generated in MATLAB fixing \texttt{rng(1)}. We use the following sparsity patterns:
\begin{equation}
P_2=\left[ \begin{array}{c c c}
 *   & *  &   0\\
 0   &  *  &   0\\
  0  &   0   &  *
\end{array} \right], P_1=\left[ \begin{array}{c c c}
 *   & 0  &  *\\
 0   &  *  &   0\\
  0  &   0   &  *
\end{array} \right], P_0=\left[ \begin{array}{c c c}
 0  & 0  & *\\
 0  &  *  &  0\\
 *  &  0  &  0
\end{array} \right],
\end{equation}
and projecting the random matrices of these sparsity patterns, we obtain respectively $A_2,A_1,A_0$:
\begin{align*}
A_2&=\left[ \begin{array}{c c c}
-6.4901 \times 10^{-1}  & -1.1096 & 0\\
 0 & -8.4555 \times 10^{-1} &  0\\
0 & 0 & -1.9686\times 10^{-1}
\end{array} \right] \\
A_1&=\left[ \begin{array}{c c c}
5.8644 \times 10^{-1} &  0 & 1.6681\times 10^{-1}\\
  0 & 8.7587 \times 10^{-1} & 0\\
  0 & 0 & -1.2701
 \end{array} \right],\\
A_0&=\left[ \begin{array}{c c c}
   0 & 0 & -1.8651\\
   0  & 1.7813 &  0\\
  -2.7516 \times 10^{-1}  & 0  & 0
 \end{array} \right].
\end{align*}
The method requires $m=23$ points and Algorithm \ref{Alg:distance to sing functions} does not need to change $m$ during the performance. Following the notations of the paper, we choose $\mbox{tol}=10^{-12}$ in Algorithm \ref{Alg:choice of m}, threshold $\beta=10^{-3}$, $\mbox{tol}_1= m\cdot 10^{-9}$ and $\mbox{tol}_2=10^{-6}$ in Algorithm \ref{Alg:distance to sing functions}. The distance to singularity for $\mathcal{A}(\lambda)$ with the sparsity constraints gives an approximation of $\varepsilon_{\mathcal{S}}=3.0032\times 10^{-1}$ and the computation ends in $15$ iterations. The maximum of the determinant of the perturbed matrix-valued function on the unit disk is $7.3342 \times 10^{-5}$ and the minimum is $2.3567 \times 10^{-6}$. 
\end{ex}

\subsection{Special case: matrix polynomials}
\label{subsec:special case_matrixpol}

In this Subsection, we compare the novel approach with other existing methods. In particular, we test our method with the one proposed in \cite{GnazzoGugl}, for the case of numerical approximation of the structured distance to singularity for matrix polynomials. For this comparison, see Examples \ref{ex:mirror} and \ref{ex:damped_beam}. Since the methods in \cite{ByersHeMehr}, for matrix pencils, and in \cite{GiesHaral,DasBora}, for matrix polynomials, are able to compute the unstructured distance to singularity, we use them as benchmark and show that the performance of our method matches their results, as shown in Examples \ref{ex:comparison_bora_cubic}, \ref{ex:comparison_haraldson} and \ref{ex:comparison_byers}. In addition, in Examples \ref{ex:mirror} and \ref{ex:damped_beam}, we show that the method does not benefit from choices of the number of points $m$, that are different from the one proposed by the novel approach. In particular, we show that choosing a fixed number of points $\tilde{m} >  m$ leads to an increase of the CPU time, without bringing relevant improvement in the precision of the method.

For Examples \ref{ex:mirror} and \ref{ex:damped_beam},  the numerical implementation of the method has been done considering the functional \eqref{eq:scaled_functional_widetildeG}. Moreover, we employ the normalization on the matrix coefficients described at the beginning of Section \ref{sec:Computational issues}.

\begin{ex}
\label{ex:mirror}
We consider the example \texttt{mirror} in the \texttt{nlevp} package. This is a quartic matrix polynomial of size $9 \times 9$. We compute the approximate distance to singularity taking as additional structure the sparsity pattern induced by the matrix coefficients.

We use the functional $\widetilde{G}_{\varepsilon}$ asking for a tolerance $\mbox{tol}_1$ in the order of $10^{-11}$ on it. It is possible to prove, using a symbolic tool for the computation of the determinant of the matrix polynomial, that in this case the degree of the determinant is $27$. In Table \ref{tab:Comparison_mirror}, we summarize the results obtained comparing the new methodology and the approach with a fixed amount of points, including the choice of $m$ suggested by the fundamental theorem of the algebra. More in details, we compare the approximated distance to singularity, the CPU time, the number of iterations needed for reaching the required accuracy. Moreover, both considering a variable choice of the number of points and a fixed one, we evaluate the 
\begin{equation*}
   \sigma_{\min}(\lambda_j):=\sigma_{\min} \left( \mathcal{Q}(\lambda_j) + \Delta \mathcal{Q}(\lambda_j) \right),
\end{equation*}
where $\mathcal{Q}(\lambda)$ is the original matrix polynomial and $\Delta \mathcal{Q}(\lambda)$ is the perturbation obtained with the methods, which makes the polynomial numerically singular, and $\lambda_j:=x_j + iy_j$ are the pairs in the intersection between 
$\left\lbrace \lambda_j : x_j \in \mbox{\texttt{[-1:0.01:1]}}, y_j \in \mbox{\texttt{[-1:0.01:1]}} \right\rbrace$ and the unit complex disk $\left\lbrace \lambda_j \in \mathbb{C}: \right. \allowbreak \left. \left| \lambda_j \right| \leq 1 \right\rbrace$. This chosen set of points $\lambda_j$ is made of $31417$ different complex points.

\begin{table}[!h]
   \caption{Comparison for the {\tt mirror} matrix polynomial.}
    \centering
    \begin{tabular}{|c|c|c|c|c|}
    \hline
         & Novel Approach  & Th. Algebra & Choice 1 & Choice 2\\
         \hline \hline
        Distance & $3.8548 \times 10^{-4}$  & $3.7441 \times 10^{-4}$ & $3.8573 \times 10^{-4}$ & $3.7832 \times 10^{-4}$\\
        \hline
        Num. points & $12$ & $28$ & $20$ & $25$\\
        \hline
        Time & $36.5683$ & $228.1456$ & $134.0780$ & $198.0705$\\
        \hline
        Iter. & $5$ & $6$ & $6$ & $6$\\
        \hline
        Max. $\sigma_{\min}(\lambda_j)$ & $3.4846 \times 10^{-5}$ & $7.3581 \times 10^{-5}$ & $6.2990 \times 10^{-5}$ & $6.9922 \times 10^{-5}$\\
        \hline
    \end{tabular}
    \label{tab:Comparison_mirror}
\end{table}
We impose as tolerance ${\rm{tol}}=10^{-12}$ in Algorithm \ref{Alg:choice of m} and $\mbox{tol}_2=10^{-6}$, threshold $\beta=10^{-3}$ in Algorithm \ref{Alg:distance to sing functions} of Section \ref{sec:Computational issues}.

The results in Table \ref{tab:Comparison_mirror} show that a raise in the number of needed points $m$ (choosing for instance $m=20,25$) leads to a subsequent increase of the computational time, while the accuracy of the solution does not improve, as it can be noticed from the last line of the table.

It could be useful to a posteriori certify that the computed matrix polynomial is indeed close to being singular, and, therefore, the approximated distance to singularity is acceptable. To this end, we refer the reader to \cite[Section 6]{GnazzoGugl}, where a posteriori upper bound checks if the computed matrix polynomial is acceptable as numerically singular. In this example, for instance, this criterion assures that the polynomial $\mathcal{Q}(\lambda) + \Delta \mathcal{Q}(\lambda)$, computed via the novel approach, has a distance equal to $3.7459\times 10^{-5}$.
\end{ex}

\begin{ex}
    \label{ex:damped_beam}
We consider the quadratic matrix polynomial \texttt{damped\_beam}, taken from the \texttt{nlevp} collection \cite{NLEVP}. This quadratic eigenvalue problem arises in the vibration analysis of a beam and it is scalable. Here we consider coefficients of size $20 \times 20$ and compute the distance to singularity with the additional constraint of the sparsity pattern. As in the previous examples, we normalize the initial coefficients of the polynomial accordingly to the proposal in Section \ref{sec:Computational issues}. The parameters employed in Algorithm \ref{Alg:distance to sing functions} are the ones provided in Example \ref{ex:mirror}. Table \ref{tab:Comparison_mirror} collects the results of the numerical comparison.

\begin{table}[h!]
  \caption{Comparison for {\tt damped\_beam}.}
    \centering
    \begin{tabular}{|c|c|c|c| c|}
    \hline
     & Novel Approach & Th. Algebra & Choice 1 & Choice 2 \\
     \hline \hline
     Distance & $ 3.8974 \times 10^{-3}$ & $ 3.8209 \times 10^{-3}$  & $3.8671 \times10^{-3}$ & $ 3.8378 \times 10^{-3}$\\
     \hline
     Num. points & $5$ & $41$ & $15$ & $30$ \\
     \hline
     Time & $66.0911$ & $409.6935$ & $154.6446$  & $295.5600$\\
     \hline
     Iter. & $9$ & $8$ & $9$ & $8$\\
     \hline
     Max. $\sigma_{\min}(\lambda_j)$ & $2.7397 \times 10^{-5}$ & $6.9578 \times 10^{-5}$ & $4.4097 \times 10^{-5}$ & $6.0273 \times 10^{-5}$\\
     \hline
    \end{tabular}
    \label{tab:Comparison_dampedbeam}
\end{table}
\end{ex}

From the results provided in Examples \ref{ex:mirror} and \ref{ex:damped_beam}, we observe that the CPU time is noticeably lower if we use the novel method. Note that the precision of both the methodologies is comparable, even if a slightly smaller upper bound for the distance to singularity is obtained by applying the method that exploits the fundamental theorem of the algebra. This decrease of the elapsed time still holds in situations where the determinant of the considered matrix polynomial has not the maximum possible degree, as shown in Example \ref{ex:mirror}.

\begin{ex}
\label{ex:comparison_bora_cubic}
In order to test the reliability of the novel approach for the matrix polynomials, we compare with the method in \cite{DasBora}. Consider the cubic matrix polynomial
\begin{small}
\[
\lambda^3 \begin{bmatrix}
   -1.9867 & 1.28 \\
    0.6097 & -0.1477
\end{bmatrix} + \lambda^2 \begin{bmatrix}
    0.6346  & 0.9689\\
     0.6252  & -0.0649
\end{bmatrix} + \lambda \begin{bmatrix}
   0.8837 & 0.9969 \\
 0.219  & 0.0259 
\end{bmatrix} + \begin{bmatrix}
    -0.1414  & -0.149\\
     1.1928  & 0.9702
\end{bmatrix},
\]
\end{small}
taken from Example $9.3$ in \cite{DasBora}. We run the method in Algorithm \ref{Alg:distance to sing functions}, choosing parameters $\beta = 10^{-3}$, $\mbox{tol}_1 = 10^{-12}$, $\mbox{tol}_2 =10^{-6}$ and $\mbox{tol}_3 =10^{-12}$. In order to compare the results we do not impose additional structures and compute the unstructured distance to singularity. Using $m=7$, our method produces an approximate distance to singularity equal to $1.676540378893858$, which is coherent with the results in Table $3$ in \cite{DasBora}, obtained using BFGS and \texttt{globalsearch.m}. Moreover, computing the values $\sigma_{\min}(\lambda_j)$ on the grid of points proposed in Examples \ref{ex:mirror} and \ref{ex:damped_beam}, we obtain that the maximum value is $6.0369 \times 10^{-6}$ and the minimum is $1.4914 \times 10^{-8}$.
\end{ex}

\begin{ex}
\label{ex:comparison_haraldson}
We consider the matrix polynomial
\begin{multline*}
\lambda^2 \begin{bmatrix}
-0.0376 & 0.107 & 0.293\\
0.003 & -0.14914 & -0.2859\\
0.0577 & 0.1455 & 0.231
\end{bmatrix} + \lambda \begin{bmatrix}
    -0.2122 & 0.363 & -0.1385\\
0.18027 &-0.151 & 0.469\\
-0.106 & 0.212 & -0.1514
\end{bmatrix} + \\ + \begin{bmatrix}
0.0278 & 0.0563 & 0.1141\\
-0.1758 & 0.327 & -0.173\\
-0.056 & 0.0321 & -0.075
\end{bmatrix},
\end{multline*}
proposed \cite[Section $5.2$]{GiesHaral} and \cite[Example $9.8$]{DasBora}, and compute an approximation of the unstructured distance to singularity. Setting the parameters as in Example \ref{ex:comparison_bora_cubic}, the novel approach requires $m=7$ and produces an approximate distance to singularity equal to $2.660288767643578 \times 10^{-2}$. The distance is coherent with the one proposed in \cite{GiesHaral} and \cite{DasBora}, since it coincides with them up to the fifth decimal digit. The computation of $\sigma_{\min}(\lambda_j)$ over the usual grid pf points produces a maximum value equal to $5.3371 \times 10^{-5}$ and a minimum value equal to $9.0167 \times 10^{-8}$. 
\end{ex}

\begin{ex}
\label{ex:comparison_byers}
    We consider the matrix pencil proposed in \cite[Example 5]{ByersHeMehr}, that is the pencil $B_n - \lambda B_n$, where
    \[
    B_n = \begin{bmatrix}
        1 & -1 & -1 & \cdots & -1 \\
        & 1 & -1 & -1 & \cdots \\
         &  & \ddots & \vdots & \vdots \\
         & & & 1 & -1 \\
         & & & & 1
    \end{bmatrix} \in \mathbb{R}^{n \times n}.
    \]
The matrix $B_n$ is an ill-conditioned triangular matrix \cite{GoluVanl13}. We choose $n=4$ and run our approach on the pencil, without imposing any additional structure, in order to be able to compare with \cite{ByersHeMehr}. Our method employs $5$ points and provides an approximation of the unstructured distance to singularity for the pencil in $14$ iterations. For this example, in \cite{ByersHeMehr}, the authors provide a closed formula for the distance to singularity, which provides the result $2.582980 \times 10^{-1}$ (up to $6$ digits). Our approach computes an approximate distance of $2.582949 \times 10^{-1}$, which is coherent with the exact one, since the distance of our solution to the exact one is in the order of $10^{-6}$.
\end{ex}

\subsection{Further implementation strategies}
\label{subsec:further_implementation_strategies}

As described in Section \ref{sec:Computational issues}, the numerical implementation of our approach may lead to different challenges and requires a careful study. Indeed, alternative proposals may lead to a speed-up of the method. In this Subsection, we provide a few further implementation strategies that could improve the method. However, an optimised version of our code is currently in progress.

\emph{Experimental choice of the number of points $m$}: as described in Subsection \ref{subsec:experimental_points}, we identify two different ways of updating the choice of the number of points $m$ in our approach. Consider again Example \ref{ex:different_tau}, with $\tau= 1$. Figure \ref{fig:Plot_points_m(eps)} shows that the number of considered points may vary in the set $\left\lbrace 16, 17, 18, 19 \right\rbrace$. We compare the possible approaches for the update of $m$ at each step of the method. The choice \emph{Half-Double} does not lead to a change of the number of points, instead, the approach \emph{Subsequent} changes the number of points $m$ three times during the process. The results obtained with these proposals are the same in terms of computed distance, number of employed iteration and accuracy of the solution (measured as in the previous examples on a prescribed set of points $\lambda_j$).

\begin{table}[h!]
    \caption{Comparison for update strategies for $m$, for Example \ref{ex:different_tau}.}
    \centering
    \begin{tabular}{|c|c|c|}
    \hline
         & Half-Double & Subsequent \\
         \hline
         \hline
       Distance  & $0.7034$ & $0.7034$ \\
       \hline
       Maximum $m$  & $19$ &  $19$ \\
       \hline
       Time  & $103.9930$ & $94.8082$\\
       \hline
       Iter. & $19$ & $19$ \\
       \hline
       Max. $\sigma_{\min}$  & $1.6654\times 10^{-5}$ & $1.6495\times 10^{-5}$\\
       \hline
       Min. $\sigma_{\min}$  & $1.7759 \times 10^{-10}$ & $3.2269\times 10^{-9}$\\
       \hline
    \end{tabular}
    \label{tab:different_update_m}
\end{table}

\emph{Inner iteration with} \texttt{manopt}: in the current version of the inner iteration, we find the local minimizers of the functional computing the stationary points of the ODE system \eqref{eq:gradient_system}, using the characterization in Theorem \ref{thm:charact_minim}. As described in Section \ref{sec:Computational issues}, we employ an explicit Euler method. Here, we explore an alternative proposal for the numerical implementation, using the Matlab package \texttt{manopt}, a toolbox for Riemannian optimization \cite{Boumal}. In particular, it is possible to substitute the gradient system approach described in Section \ref{subsec:Inner for matrix-valued functions} with a more sophisticated solver, such as the Riemannian trust-region method available in \texttt{manopt}. We tested the behavior of this proposal, comparing the results obtained for Example \ref{ex:time delay nlevp}, asking only for real perturbations $\Delta_i$, without any additional structure.

\begin{table}[!h]
 \caption{Comparison of the method among the approach in Subsection \ref{subsec:Inner for matrix-valued functions} and {\tt manopt} for Example \ref{ex:time delay nlevp}.}
    \centering
  \begin{tabular}{|c|c|c|}
\hline
     & ODE & Manopt \\
     \hline \hline
  Distance   & $0.0428$  & $0.0417$ \\
  \hline
    Iter & $22$ & $25$ \\
    \hline
    Max. $\sigma_{\min}$ & $1.9875 \times 10^{-5}$ & $ 2.0274\times 10^{-5}$\\
      \hline
    Min. $\sigma_{\min}$ & $4.3760 \times 10^{-9}$ & $7.4665 \times 10^{-9}$\\ 
    \hline
\end{tabular}
    \label{tab:my_label}
\end{table}

Our (non-optimized) code version with the Riemannian trust-region method provides a computed distance to singularity comparable to one obtained with the ODE approach. The elapsed time for the \texttt{manopt} implementation is roughly half that the ODE method employs. This may suggest that an optimization and a speed-up of our current code is possible, and using second-order methods could improve the performance of our technique.

\subsection{Robustness of the approach}
\label{subsec:robustness}

As anticipated in Subsection \ref{subsec:robustness_theory}, we may include an additional verification of the robustness of the proposed approach, taking into account the extra check on a prescribed set of points $\Omega$. In the numerical implementation of our method, we perform the computation of the quantity in \eqref{eq:mean}, by considering discrete sets $\Omega_k$, for $k=1,2,\ldots$ for which the cardinalities satisfy $\left| \Omega_{k+1} \right| = 2\left|\Omega_k \right|$, and stopping the procedure when the quantity \eqref{eq:mean} for $\Omega_k$ and $\Omega_{k+1}$ are comparable.

Consider again Example  \ref{ex:different_tau}, for $\tau =1$. We run the approach, checking whether the error on the discrete set of points $\Omega$ is lower or equal to the tolerance tol. We perform this additional check at the steps for which the re-computation of the points is needed, following the approach described in Algorithm \ref{Alg:choice of m}. In Figure \ref{fig:check_discr}, we plot the quantity $b_m$ in \eqref{eq:mean} and the quantity $\left| a_{m}(\mathbf{F}_k(\lambda))\right|$ in \eqref{eq:update_mtilde}, computed every time that we need to re-compute the number of points, comparing them with the required tolerance $\mbox{tol}=10^{-12}$. In this Example, the quantity \eqref{eq:update_mtilde} and the quantity \eqref{eq:mean} provide a similar behavior, since they are always of the same order of magnitude (lower than the chosen tolerance). 

Another possibility is replacing the stopping criterion \eqref{eq:update_mtilde} with a different one, relying on the quantity \eqref{eq:mean}. In detail, we could choose to monitor the approximation error using the value in \eqref{eq:mean} and asking that the number of selected points $\tilde{m}$ is the smallest $m$ such that
\[
    \max_{\omega \in \Omega} \left|p_m(\omega) - f(\omega) \right| \leq \mbox{tol},
\]
where $p_m$ is the interpolating polynomial of $f(\lambda)$ at the points $\left\lbrace\mu_{j}\right\rbrace_{j=1}^m$. We consider this possibility in the first iteration of Example \ref{ex:different_tau}, with $\tau=1$, and compare the result with the proposal in \eqref{eq:update_mtilde}. Figure \ref{fig:2strategies} provides a comparison of these two strategies. In this Example, the behavior is similar, suggesting the choice $m=19$ both for the stopping criterion in \eqref{eq:update_mtilde} and the one in \eqref{eq:mean}.

\begin{figure}[h!]
    \centering
\begin{tikzpicture}
    \begin{axis}[
            legend pos = north west,
            height=6cm,
            xlabel= iterations,
            width=\linewidth,
            xmin=1, xmax=14,
            ]
              \addplot[blue, dashed] table[y index = 1]{M162529_example_delay.dat};
    \addplot[red, dashed] table[y index = 2]{M162529_example_delay.dat};
   \addplot[green, dashed] table[y index = 3]{M162529_example_delay.dat};
 \addplot[only marks, blue, mark size =1pt] table[y index = 1]{M162529_example_delay.dat};
  \addplot[only marks, red, mark size =1pt] table[y index = 2]{M162529_example_delay.dat};
   \legend{\small{Stopping criterion \eqref{eq:update_mtilde} },\small{Stopping criterion \eqref{eq:mean}}, \small{Tol}}
    \end{axis}    
\end{tikzpicture}
    \caption{Additional check on the robustness of the approach, for Example \ref{ex:different_tau}, with $\tau=1$.}
    \label{fig:check_discr}
\end{figure}

\begin{figure}[h!]
    \centering
\begin{tikzpicture}
\begin{semilogyaxis}[
legend pos = north east,
xmin=5, xmax=19,
height= 5cm,
xlabel= number of points $m$,
width=\linewidth,
]
 \addplot[red, dashed] table[y index = 1]{M162529_decaym_new2.dat}; 
     \addplot[blue, dashed] table[y index = 1]{M162529_decayd.dat}; 
     \addplot[only marks, blue, mark size =1pt]table[y index = 1]{M162529_decayd.dat}; 
    \addplot[only marks, red, mark size =1pt]table[y index = 1]{M162529_decaym_new2.dat}; 
    \legend{\small{Stopping criterion \eqref{eq:mean}},\small{Stopping criterion \eqref{eq:update_mtilde}}}
\end{semilogyaxis}
\end{tikzpicture}
\caption{Values of the stopping criteria in \eqref{eq:update_mtilde} and \eqref{eq:mean}, for the first step of the method in Example \ref{ex:different_tau}, with $\tau=1$.}
    \label{fig:2strategies}
\end{figure}

\section*{Conclusions}
We presented a method for the numerical approximation of the distance to singularity for matrix-valued functions. Taking inspiration from the method in \cite{GnazzoGugl}, we propose a method for the construction of a minimization problem relying on a discrete set of points. The possible presence of an infinite number of eigenvalues for a general matrix-valued function $\mathcal{F}(\lambda)$ represents a delicate feature of the problem and our proposal consists in exploiting the maximum modulus principle. One of the highlights of our approach is the possibility to extend it to structured perturbations, with a few changes in the algorithm. Moreover, this approach provides new possibilities for the approximation of the distance to singularity for matrix polynomials, making the computation cheaper to perform.

\section*{Acknowledgements}
We wish to thank two anonymous Referees for their helpful suggestions and constructive remarks. Miryam Gnazzo and Nicola Guglielmi are affiliated to the Italian INdAM-GNCS (Gruppo Nazionale di Calcolo Scientifico). Furthermore, during part of the preparation of this work, Miryam Gnazzo was affiliated with Gran Sasso Science Institute, L'Aquila (Italy). Nicola Guglielmi acknowledges that his research was supported by funds from the Italian 
MUR (Ministero dell'Universit\`a e della Ricerca) within the 
PRIN 2022 Project ``Advanced numerical methods for time dependent parametric partial differential equations with applications''.

\nocite{*}
\bibliographystyle{siamplain}
\bibliography{M162529_biblio}

\begin{thebibliography}{10}

\bibitem{Rodriguez}
{\sc U.~Alvarez-Rodriguez, A.~Perez-Leija, I.~L. Egusquiza, M.~Gräfe, M.~Sanz,
  L.~Lamata, A.~Szameit, and E.~Solano}, {\em Advanced-retarded differential
  equations in quantum photonic systems}, Sci. Rep., 7, 42933 (2017),
  \url{https://doi.org/10.1038/srep42933}.

\bibitem{DopicoZaballa}
{\sc A.~Amparan, F.~M. Dopico, S.~Marcaida, and I.~Zaballa}, {\em Strong
  linearizations of rational matrices}, SIAM J. Matrix Anal. Appl., 39 (2018),
  pp.~1670--1700, \url{https://doi.org/10.1137/16M1099510}.

\bibitem{AustinTref2014}
{\sc A.~P. Austin, P.~Kravanja, and L.~N. Trefethen}, {\em Numerical algorithms
  based on analytic function values at roots of unity}, SIAM J. Numer. Anal.,
  52 (2014), pp.~1795--1821, \url{https://doi.org/10.1137/130931035}.

\bibitem{NLEVP}
{\sc T.~Betcke, N.~J. Higham, V.~Mehrmann, C.~Schr\"{o}der, and F.~Tisseur},
  {\em {NLEVP}: A collection of nonlinear eigenvalue problems}, ACM Trans.
  Math. Softw., 39 (2013), \url{https://doi.org/10.1145/2427023.2427024}.

\bibitem{Boumal}
{\sc N.~Boumal}, {\em An introduction to optimization on smooth manifolds},
  Cambridge University Press, Cambridge, 2023,
  \url{https://doi.org/10.1017/9781009166164}.

\bibitem{breda2023practical}
{\sc D.~Breda and D.~Liessi}, {\em A practical approach to computing {L}yapunov
  exponents of renewal and delay equations}, Math. Biosci. Eng., 21 (2024),
  pp.~1249--1269, \url{https://doi.org/10.3934/mbe.2024053}.

\bibitem{ByersHeMehr}
{\sc R.~Byers, C.~He, and V.~Mehrmann}, {\em Where is the nearest non-regular
  pencil?}, Linear Algebra Appl., 285 (1998), pp.~81--105,
  \url{https://doi.org/10.1016/S0024-3795(98)10122-2}.

\bibitem{Cartan}
{\sc H.~Cartan}, {\em Elementary Theory of Analytic Functions of One Or Several
  Complex Variables}, Dover Publications, New York, 1995.

\bibitem{CooleyTukey}
{\sc J.~W. Cooley and J.~W. Tukey}, {\em An algorithm for the machine
  calculation of complex fourier series}, Math. Comp., 19 (1965), pp.~297--301,
  \url{https://www.ams.org/journals/mcom/1965-19-090/S0025-5718-1965-0178586-1/}.

\bibitem{Corless}
{\sc R.~M. Corless, G.~H. Gonnet, D.~E.~G. Hare, D.~J. Jeffrey, and D.~E.
  Knuth}, {\em On the {L}ambert{$W$} function}, Adv. Comput. Math., 5 (1996),
  p.~329–359, \url{https://doi.org/10.1007/bf02124750}.

\bibitem{DasBora}
{\sc B.~Das and S.~Bora}, {\em Nearest rank deficient matrix polynomials},
  Linear Algebra Appl., 674 (2023), pp.~304--350,
  \url{https://doi.org/10.1016/j.laa.2023.05.019}.

\bibitem{DopicoNoferiniNyman}
{\sc F.~M. Dopico, V.~Noferini, and L.~Nyman}, {\em A {R}iemannian optimization
  method to compute the nearest singular pencil}, SIAM J. Matrix Anal. Appl.,
  45 (2024), pp.~2007--2038, \url{https://doi.org/10.1137/23M1596326}.

\bibitem{GiesHaral}
{\sc M.~Giesbrecht, J.~Haraldson, and G.~Labahn}, {\em Computing the nearest
  rank-deficient matrix polynomial}, in Proceedings of the 2017 ACM on
  International Symposium on Symbolic and Algebraic Computation, 2017,
  p.~181–188, \url{https://doi.org/10.1145/3087604.3087648}.

\bibitem{GnazzoGugl}
{\sc M.~Gnazzo and N.~Guglielmi}, {\em Approximating the closest structured
  singular matrix polynomial}, Linear Multilinear Algebra,  (2024), pp.~1--29,
  \url{https://doi.org/10.1080/03081087.2024.2376561}.

\bibitem{Robol}
{\sc M.~Gnazzo and L.~Robol}, {\em Backward errors for multiple eigenpairs in
  structured and unstructured nonlinear eigenvalue problems}, {\rm{preprint}},
  (2024), \url{https://arxiv.org/abs/2405.06327}.

\bibitem{GoluVanl13}
{\sc G.~H. Golub and C.~F. Van~Loan}, {\em Matrix Computations}, The Johns
  Hopkins University Press, Baltimore, 4th~ed., 2013.

\bibitem{GH99}
{\sc N.~Guglielmi and E.~Hairer}, {\em Order stars and stability for delay
  differential equations}, Numer. Math., 83 (1999), pp.~371--383,
  \url{https://doi.org/10.1007/s002110050454}.

\bibitem{GH01}
{\sc N.~Guglielmi and E.~Hairer}, {\em Implementing {R}adau {IIA} methods for
  stiff delay differential equations}, Computing, 67 (2001), pp.~1--12,
  \url{https://doi.org/10.1007/s006070170013}.

\bibitem{GH08}
{\sc N.~Guglielmi and E.~Hairer}, {\em Computing breaking points in implicit
  delay differential equations}, Adv. Comput. Math., 29 (2008), pp.~229--247,
  \url{https://doi.org/10.1007/s10444-007-9044-5}.

\bibitem{GuglLubich}
{\sc N.~Guglielmi and C.~Lubich}, {\em Differential equations for roaming
  pseudospectra: Paths to extremal points and boundary tracking}, SIAM J.
  Numer. Anal., 49 (2011), pp.~1194--1209,
  \url{https://doi.org/10.1137/100817851}.

\bibitem{GugLubMeh}
{\sc N.~Guglielmi, C.~Lubich, and V.~Mehrmann}, {\em On the nearest singular
  matrix pencil}, SIAM J. Matrix Anal. Appl., 38 (2017), pp.~776--806,
  \url{https://doi.org/10.1137/16M1079026}.

\bibitem{Sicilia}
{\sc N.~Guglielmi, C.~Lubich, and S.~Sicilia}, {\em Rank-1 matrix differential
  equations for structured eigenvalue optimization.}, SIAM J. Numer. Anal., 61
  (2023), pp.~1737--1762, \url{https://doi.org/10.1137/22M1498735}.

\bibitem{GutTis}
{\sc S.~Güttel and F.~Tisseur}, {\em The nonlinear eigenvalue problem}, Acta
  Numer., 26 (2017), p.~1–94,
  \url{https://doi.org/10.1017/S0962492917000034}.

\bibitem{HW96}
{\sc E.~Hairer and G.~Wanner}, {\em Solving ordinary differential equations.
  {II}}, vol.~14 of Springer Series in Computational Mathematics,
  Springer-Verlag, Berlin, revised~ed., 2010,
  \url{https://doi.org/10.1007/978-3-642-05221-7}.

\bibitem{HornJoh}
{\sc R.~A. Horn and C.~R. Johnson}, {\em Matrix Analysis}, Cambridge University
  Press, 1990.

\bibitem{JarMich}
{\sc E.~Jarlebring and W.~Michiels}, {\em Invariance properties in the root
  sensitivity of time-delay systems with double imaginary roots}, Automatica,
  46 (2010), pp.~1112--1115,
  \url{https://doi.org/10.1016/j.automatica.2010.03.014}.

\bibitem{Luenberger}
{\sc D.~G. Luenberger and Y.~Ye}, {\em Linear and Nonlinear Programming},
  Springer, Cham, 5th~ed., 2021,
  \url{https://doi.org/10.1007/978-3-030-85450-8}.

\bibitem{MichNiculescu}
{\sc W.~Michiels and S.~Niculescu}, {\em Stability and Stabilization of
  Time-Delay Systems}, SIAM, Philadelphia, 2007,
  \url{https://doi.org/10.1137/1.9780898718645}.

\bibitem{Myshkis}
{\sc A.~D. Myshkis}, {\em Mixed functional differential equations}, J. Math.
  Sci., 129 (2005), p.~4111–4226,
  \url{https://doi.org/10.1007/s10958-005-0345-2}.

\bibitem{TrefWeid}
{\sc L.~N. Trefethen and J.~A.~C. Weideman}, {\em The exponentially convergent
  trapezoidal rule}, SIAM Rev., 56 (2014), pp.~385--458,
  \url{https://doi.org/10.1137/130932132}.

\end{thebibliography}

\end{document}